\documentclass[reqno, 11pt]{amsart}
\usepackage{url}
\usepackage[colorlinks]{hyperref}
\usepackage{helvet}
\usepackage{mathrsfs}

\usepackage[usenames,dvipsnames]{xcolor}

\newtheorem{theorem}{Theorem}[section]
\newtheorem{lemma}[theorem]{Lemma}

\newtheorem{proposition}[theorem]{Proposition}
\newtheorem{corollary}[theorem]{Corollary}

\theoremstyle{definition}

\newtheorem{remark}[theorem]{Remark}

\numberwithin{equation}{section}

\usepackage{bbm}
\usepackage{url}
\usepackage{euscript}
\usepackage{amsmath}
\usepackage{amsthm}

\usepackage{amsxtra}
\usepackage{amssymb}
\usepackage{pifont}
\usepackage{amsbsy}

\usepackage{graphicx}
\usepackage{epstopdf}

\newskip\aline \newskip\halfaline
\def\skipaline{\vskip\aline}

\def\qedbox{$\rlap{$\sqcap$}\sqcup$}
\def\qed{\nobreak\hfill\penalty250 \hbox{}\nobreak\hfill\qedbox\skipaline}

\newcommand{\one}{{\mathbbm{1}}}

\newcommand{\bA}{\mathbb A}

\newcommand\bC{{\mathbb C}}

\newcommand\bE{{\mathbb E}}

\newcommand{\bN}{{{\mathbb N}}}

\newcommand{\bP}{{{\mathbb P}}}

\newcommand\bR{{\mathbb R}}

\newcommand{\bT}{{\mathbb T}}

\newcommand\bZ{{\mathbb Z}}

\DeclareMathOperator{\supp}{{\rm supp}}

\DeclareMathOperator{\var}{\boldsymbol{var}}

\DeclareMathOperator{\vol}{vol}

\newcommand{\ba}{\boldsymbol{a}}

\newcommand{\be}{{\boldsymbol{e}}}

\newcommand{\ii}{\boldsymbol{i}}

\newcommand{\kk}{{\boldsymbol{k}}}

\newcommand{\bp}{{\boldsymbol{p}}}

\newcommand{\bs}{\boldsymbol{s}}
\newcommand{\bt}{\boldsymbol{t}}
\newcommand{\bu}{{\boldsymbol{u}}}
\newcommand{\bv}{{\boldsymbol{v}}}

\newcommand{\bx}{{\boldsymbol{x}}}
\newcommand{\by}{{\boldsymbol{y}}}
\newcommand{\bz}{{\boldsymbol{z}}}

\newcommand{\bsE}{\boldsymbol{E}}

\newcommand{\bsI}{\boldsymbol{I}}
\newcommand{\bsJ}{\boldsymbol{J}}

\newcommand{\bsN}{\boldsymbol{N}}

\newcommand{\bsP}{\boldsymbol{P}}

\newcommand{\bsV}{{\boldsymbol{V}}}
\newcommand{\bsW}{\boldsymbol{W}}

\newcommand{\bsY}{{\boldsymbol{Y}}}
\newcommand{\bsZ}{\boldsymbol{Z}}

\newcommand{\bone}{\boldsymbol{1}}

\newcommand{\bgamma}{{\boldsymbol{\gamma}}}
\newcommand{\bGamma}{\boldsymbol{\Gamma}}

\newcommand{\blam}{{\boldsymbol{\lambda}}}

\newcommand{\bnu}{{\boldsymbol{\nu}}}

\newcommand{\btheta}{\boldsymbol{\theta}}

\newcommand{\bom}{{\boldsymbol{\omega}}}

\newcommand{\boxi}{\boldsymbol{\xi}}

\newcommand{\bXi}{\boldsymbol{\Xi}}
\newcommand{\bOm}{\boldsymbol{\Omega}}

\newcommand{\si}{{\sigma}}

\newcommand{\ve}{{\varepsilon}}

\newcommand{\eA}{\EuScript{A}}
\newcommand{\eB}{\EuScript{B}}

\newcommand{\eF}{\EuScript{F}}

\newcommand{\eI}{{\EuScript{I}}}

\newcommand{\eK}{\EuScript{K}}
\newcommand{\eL}{\EuScript{L}}

\newcommand{\eN}{\EuScript{N}}
\newcommand{\eO}{\EuScript{O}}
\newcommand{\eP}{\EuScript{P}}

\newcommand{\eS}{\EuScript{S}}

\newcommand{\eX}{\mathscr{X}}

\newcommand{\lan}{\langle}
\newcommand{\ran}{\rangle}

\def\inpr{\mathbin{\hbox to 6pt{\vrule height0.4pt width5pt depth0pt \kern-.4pt \vrule height6pt width0.4pt depth0pt\hss}}}

\newcommand{\pa}{\partial}

\newcommand{\whB}{\widehat{B}}
\newcommand{\tbY}{\widehat{\boldsymbol{Y}}}

\newcommand{\mfH}{{\mathfrak{H}}}

\begin{document}

\title[Multidimensional  random Fourier series]{Critical points  of  multidimensional  random Fourier series: central limits} 

\date{Started  October 30, 2015. Completed  on  November 8, 2015. Last modified on {\today}. }

\author{Liviu I. Nicolaescu}
%\thanks{These notes are for myself and whoever else reads this footnote.}

\address{Department of Mathematics, University of Notre Dame, Notre Dame, IN 46556-4618.}
\email{nicolaescu.1@nd.edu}
\urladdr{\url{http://www.nd.edu/~lnicolae/}}


\begin{abstract}   We investigate certain families  $X^\hbar$, $0<\hbar \ll 1$ of stationary  Gaussian random  smooth  functions  on the $m$-dimensional torus $\bT^m:=\bR^m/\bZ ^m$ approaching the white noise as $\hbar \to 0$. We show  that  there exists universal  constants $c_1,c_2>0$ such that for any cube $B\subset \bR^m$ of size $r\leq 1/2$ and centered at the origin, the number of critical points   of $X^\hbar$ in the region $B\bmod \bZ^m\subset\bT^m$    has  mean $\sim c_1\vol(B)\hbar^{-m}$, variance $\sim c_2\vol(B)\hbar^{-m}$, and satisfies a central limit theorem as $\hbar\searrow 0$.
\end{abstract}
\maketitle

\tableofcontents

\section{The main results}
\setcounter{equation}{0}

\subsection{By way of motivation} During the final years  of his life, V.I. Arnold \cite{Ar06, Ar07a, Ar07b, Ar15} investigated the Morse theoretic properties of  trigonometric polynomials on  higher dimensional tori.      We describe below one direction of his investigation.

Fix a convex polygon $\eP$ in $\bR^m$, with vertices in the lattice $\bZ^m$.  We assume that $\eP$ symmetric with respect to the origin.     For $\btheta=(\theta_1,\dotsc, \theta_m) \in\bR^m$ and $\kk\in\bZ^m$ we set
\[
 \lan\kk,\btheta\ran :=\sum_jk_j\theta_j. 
 \]
   Consider the   vector space $\bsV_\eP$ of trigonometric polynomials on $\bT^m=\bR^m/\bZ^m$  whose Newton's polyhedra are contained in  $\eP$, i.e.,  polynomials of  the form
\begin{equation}\label{r_trig}
f(\btheta)=\sum_{\kk\in \eP\cap\bZ^m}\bigl(\, A_{\kk}\cos(2\pi\lan\kk,\btheta\ran)+B_{\kk} \sin(2\pi \lan\kk,\btheta\ran)\,\bigr),\;\;\btheta\in\bT^m.
\end{equation}
We can  think of $f$  as a random polynomial   by declaring the coefficients $A_\kk$ and $B_\kk$ to be independent  standard normal random variables.  We denote by $\bsN^\eP(f)$ the number  of  critical points of  $f(\theta)$.  We denote $\bsN_{\max}^\eP$ the (essential)  supremum of the positive random variable $\bsN^\eP(f)$. The results of Bernshtein \cite{Be} and Kouchnirenko \cite{Kou}  imply the (very rough)   upper bound
 \begin{equation}
  \bsN^\eP_{\max} \leq  K_m(\eP):=m! {\rm vol}\, (\eP),
 \label{eq: upper}
 \end{equation}
 Arnold proved  that in dimension $m=2$  and  for several classes of  polyhedra $\eP$  the upper  bound  $K_m(\eP)$ is  very  close to being optimal.   He  achieved this  by producing  nontrivial lower bounds  for  $\bsN^\eP_{\max}$ using  techniques from real algebraic geometry.
 
There exists an obvious  probabilistic lower bound  for $\bsN^\eP_{\max}$, namely
\begin{equation}\label{pr_bo}
\bsN^\eP_{\max}\geq  Z^\eP:=\bE\bigl[\,\bsN^\eP(f)\,\bigr].
\end{equation}
In \cite{N2010} we have computed $Z^\eP$ for  symmetric polyhedra $\eP$, i.e., polyhedra that are invariant under the obvious action of the symmetric group $S_m$ on $\bR^m$ and  we have observed that for the  polyhedra $\eP$ considered by Arnold the lower bound  $Z^\eP$  is so close to  the upper bound $K_m(\eP)$ that we can conclude   via simple topological arguments that $\bsN^\eP_{\max} = K_m(\eP)$.

The computations in  \cite{N2010} also show  that  for any symmetric polyhedron $\eP$  there exists  an explicit constant $C=C_m(\eP)>0$ such that
\[
Z^{\nu \eP}\sim  C_m(\eP)\nu^m\;\;\mbox{as $\nu\to\infty$ inside $\bN$}.
\]
Clearly $K_m(\nu\eP)=K_m(\eP)\nu^m$   so the  mean $Z^{\nu\eP}$ and the   Bernshtein-Kouchnirenko  upper bound $K_m(\nu\eP)$ have the same rate of growth  as $\nu\to\infty$. These observations suggest that the  random variables $\bsN^{\nu \eP}(f)$  might concentrate near  their means as $\nu\to\infty$.    

We can  reformulate the above facts as follows.  Denote by $w$   the indicator function of $\eP$. We can  describe the random trigonometric polynomial (\ref{r_trig}) in the form
\[
f(\btheta)=\sum_{\kk\in \bZ^m}w(\kk)\bigl(\, A_{\kk}\cos(2\pi\lan\kk,\btheta\ran)+B_{\kk} \sin(2\pi \lan\kk,\btheta\ran)\,\bigr).
\]
The rescaling $\eP\mapsto \nu\eP$ corresponds to a new random polynomial
\begin{equation}\label{r_trigh}
f^\hbar(\btheta)=\sum_{\kk\in \bZ^m}w(\hbar \kk)\bigl(\, A_{\kk}\cos(2\pi\lan\kk,\btheta\ran)+B_{\kk} \sin(2\pi \lan\kk,\btheta\ran)\,\bigr),\;\;\hbar=\nu^{-1}.
\end{equation}
In this paper we investigate  random Fourier series  of the form  (\ref{r_trigh}),  where $w$ is a nonnegative, \emph{radially symmetric, Schwartz function}.   (If $w$ has compact support then $f^\hbar$ is a trigonometric polynomial).  The main results  in this paper will show   that as $\hbar \to 0$ the  critical  points  of $f^\hbar(\btheta)$    will  distribute uniformly  on $\bR^m/\bZ^m$ with high probability.   In particular, the  total number of critical points of $f^\hbar(\btheta)$ inside a cube  $B\subset [0,1]^m$ satisfies a central limit theorem as $\hbar\to 0$.  The next subsections   contain the precise  statements.

\subsection{The problem.}   We begin by recalling the setup in \cite{Nifluct}. For any $\hbar>0$ we denote by $\bT^m_\hbar$  the $m$-dimensional torus $\bR^m/\bZ^m$  with  angular coordinates $\theta_1,\dotsc, \theta_m\in\bR/\bZ$  equipped with  the flat metric
\[
g_\hbar:=\sum_{j=1}^m \hbar^{-2}(d\theta_j)^2.
\]
For a measurable subset $S\subset \bT^m$ we denote by $\vol_\hbar(S)$ its volume with respect to the  metric $g_\hbar$, and we set $\vol :=\vol_\hbar|_{\hbar=1}$. Hence
\[
\vol_\hbar(\bT^m)=\hbar^{-m}\vol(\bT^m)=\hbar^{-m}.
\]
The eigenvalues of the corresponding  Laplacian $\Delta_\hbar=-\hbar^{2}\sum_{k=1}^m \pa^2_{\theta_k}$ are
\[
\lambda^\hbar(\kk)=\hbar^2\lambda(\kk),\;\;\lambda(\kk):=\bigl|\,2\pi\kk\,\bigr|^2,\;\;\kk=(k_1,\dotsc, k_m)\in\bZ^m.
\]
Denote by $\prec$ the lexicographic order on $\bR^m$.  An  \emph{orthonormal} basis of $L^2(\bT^m_\hbar)$ is given by the  functions $(\psi^\hbar_{\kk})_{\kk\in\bZ^m}$, where
\[
\psi^\hbar_{\kk}(\btheta) =\hbar^{\frac{m}{2}}\psi_\kk(\btheta),\;\;\psi_\kk(\btheta):=\begin{cases}1, &\kk=0\\
\sqrt{2} \sin 2\pi\lan \kk,\btheta\ran, &\kk\succ\vec{0}, \\
\sqrt{2}\cos 2\pi\lan\kk,\btheta\ran, &\kk\prec\vec{0}. 
\end{cases}
\]
Fix a  nonnegative, even  Schwartz function $w\in \eS(\bR)$, set $w_\hbar(t)=w(\hbar t)$ so that
\[
w\bigl(\, \sqrt{\lambda^\hbar(\kk)}\;\bigr)=w_\hbar\bigl(\,\sqrt{\lambda(\kk)}\;\bigr).
\]
 Consider the random  function given by the random Fourier series
\begin{equation}
\begin{split}
X^\hbar(\btheta)=\sum_{\kk\in\bZ^m} w_\hbar\bigl(\;\sqrt{\lambda^\hbar(\kk)}\;\bigr)^{\frac{1}{2}}\,N_{\kk}\psi_{\kk}^\hbar(\btheta)=\hbar^{\frac{m}{2}}\sum_{\kk\in\bZ^m} w_\hbar\bigl (\sqrt{\lambda(\kk)}\;\bigr)^{\frac{1}{2}}\,N_{\kk}\psi_{\kk}(\btheta)\\
=\hbar^{\frac{m}{2}}\sum_{\kk\in\bZ^m}w\bigl(2\pi\hbar|\kk|\bigr)^{\frac{1}{2}}\;\bigl(\, A_{\kk}\cos(2\pi\lan\kk,\btheta\ran)+B_{\kk} \sin(2\pi \lan\kk,\btheta\ran)\,\bigr),
\end{split}
\label{uve}
\end{equation}
where the coefficients $A_{\kk}, B_{\kk},N_{\kk}$, $\kk\in\bZ^m$, are independent standard normal random variables. 

Note that if $w\equiv 1$ in a neighborhood of $1$, then the random function $\hbar^{-m/2}X^\hbar$ converges to a Gaussian white-noise on $\bT^m$ and, extrapolating, we can think of the $\hbar\to 0$ limits  in this paper  as white-noise limits.

The random function $X^\hbar(\btheta)$ is a.s. smooth and Morse. For any Borel set $\eB\subset \bT^m$ we  denote  by $\bsZ(X^\hbar, \eB)$ the number of critical points   of $X^\hbar$ in $\eB$. In \cite{Nifluct} we have shown that   there exist constants $C=C_m(w)>0$, $S=S_m(w)\geq 0$ such that, for any open set $\eO\subset \bT^m$,
\[
\bE\bigl[\,\bsZ(X^\hbar,\eO)\,\bigr]\sim  C_m(w)\hbar^{-m}\vol(\eO)\;\;\mbox{as $\hbar\to 0$}.
\]
\[
\var[\,\bsZ(X^\hbar,\eO)\,\bigr]\sim  S_m(w)\hbar^{-m}\vol(\eO)\;\;\mbox{as $\hbar\to 0$}.
\]
In \cite{Nifluct} we described the constants $C_m(w)$  and $S_m(w)$ explicitly as certain  rather  complicated Gaussian integrals, and  we conjectured that $S_m(w)$ is actually strictly positive.

In  this paper  we will show  that indeed $S_m(w)>0$, and we will  prove a  central limit theorem  stating that, as $\hbar\to 0$, the random variables
\[
\zeta_\hbar(\eO):=\left(\frac{\hbar}{\vol(\eO)}\right)^{\frac{m}{2}}\Bigl(\, \bsZ(X^\hbar,\eO)-\bE\bigl[\,\bsZ(X^\hbar,\eO)\,\bigr]\,\Bigr)
\]
converge in law to a \emph{nondegenerate}   normal random variable $\sim\eN\bigl(\, 0, S_m(w)\,\bigr)$.  

Our approach  relies  on abstract central limit results of the type pioneered by  Breuer and Major \cite{BrMaj}.  This requires placing the the problem within a  Gaussian Hilbert space context.   To achieve this we imitate the  strategy  employed by    Aza\"{i}s and Le\'{o}n \cite{AzLe} in a related $1$-dimensional problem.

\subsection{The Wiener chaos  setup}   Let $\bx=(x_1,\dotsc, x_m)$ denote the   standard  Euclidean coordinates on $\bR^m$.  For $\bp_0\in\bR^m$ and  $R>0$ we set
\[
\whB_R(\bp_0):=\Bigl\{\bx\in\bR^m;\;\;|\bx-\bp_0|_\infty\leq \frac{R}{2}\,\Bigr\},\;\;\whB_R=\whB_R(0)=\Bigl[-\frac{R}{2},\frac{R}{2}\Bigr]^m.
\]
 For  $r\in (0,1]$ and denote by $B_{r}$ the image of the cube $\whB_{r}$ in the quotient $\bR^m/\bT^m$.  Thus, for $r<1$, $B_{r}$ is a cube on the torus centered at $0$, while $B_1=\bT^m$. 
 
 We identify the tangent space $T_0\bT^m_\hbar$ with $\bR^m$ and we denote by $\exp_\hbar$ the exponential map $\exp_\hbar: \bR^m\to\bT^m_\hbar$ defined by the metric   $g_\hbar$. In the coordinates $\bx$ on $\bR^m$ and  $\btheta$ on $\bT^m$, this map is described by $\btheta=\hbar\bx \bmod \bZ^m$. Using this map we obtain  by pullback a  $(\hbar^{-1}\bZ)^m$-periodic random  function on $\bR^m$,
 \[
 Y^\hbar(\bx):=(\exp^*_\hbar X^\hbar)(\bx)=\hbar^{m/2}\sum_{\kk\in\bZ^m}w\bigl(2\pi\hbar|\kk|\bigr)^{\frac{1}{2}}\;\bigl(A_{\kk}\cos(2\pi\hbar\lan\kk,\bx\ran)+B_{\kk} \sin(2\pi\hbar \lan\kk,\bx\ran)\bigr).
 \]
We denote  by $\bsZ(Y^\hbar, \eB)$ the number of critical points of $Y^\hbar$ in the Borel set $\eB\subset\bR^m$. Note that
\begin{equation}\label{z_resc}
\bsZ\bigl(\, X^\hbar,  B_{r}\bigr)=\bsZ\bigl(\,Y^\hbar, \whB_{\hbar^{-1}r}\,\bigr),\;\;\forall r\in (0,1].
\end{equation}
To investigate $\bsZ\bigl(\,Y^\hbar, \whB_{\hbar^{-1}r}\,\bigr)$ it is convenient  to  give  an alternate  description  to  the random function $Y^\hbar$.

A simple computation shows  that the  covariance kernel of $Y^\hbar$ is
  \begin{equation}\label{covyve}
  \begin{split}
 \eK^\hbar(\bx,\by)=\hbar^m\sum_{\kk\in\bZ^m }w_\hbar (2\pi\hbar|\kk|)\cos 2\pi\hbar\bigl\lan\, \kk, \underbrace{\by-\bx}_{\bz}\,\bigr\ran\\
 =\hbar^m\sum_{\kk\in\bZ^m }w (2\pi\hbar|\kk|)\exp\bigl(\,-2\pi\hbar\ii\lan\kk, \bz\ran\,\bigr).
 \end{split}
 \end{equation}
Define 
\[
\phi_\bz:\bR^m\to\bC,\;\;\phi_\bz(\boxi):=e^{ -\ii\lan\boxi,\bz\ran}w(|\boxi|).
\]
Using  the Poisson formula \cite[\S 7.2]{H1} we deduce  that for any $a>0$ we have
\[
\sum_{\kk\in\bZ^m}\phi_\bz\left(\frac{2\pi}{a}\kk\right)=\left(\frac{a}{2\pi}\right)^m \sum_{\bnu\in\bZ^m}\widehat{\phi_\bz}(a\bnu), 
\]
where,  for any $u=u(\boxi)\in\eS(\bR^m)$,  we denote by $\widehat{u}(\bt)$ its Fourier transform
\begin{equation}\label{fourier}
\widehat{u}(\bt)=\int_{\bR^m} e^{-\ii\lan\boxi,\bt\ran} u(\boxi)|d\boxi|. 
\end{equation}
If we let  $a=\hbar^{-1}$, then we deduce
\[
\eK^\hbar(\bx,\by)=\frac{1}{(2\pi)^m}\sum_{\bnu\in\bZ^m}\widehat{\phi_\bz}\left(\,\hbar^{-1}\bnu\,\right). 
\]
Define $V:\bR^m\to \bR$ by
\begin{equation}\label{vxi}
V(\bt): =\frac{1}{(2\pi)^m}\widehat{w}(\bt)=\frac{1}{(2\pi)^m}\int_{\bR^m} e^{-\ii\lan\boxi,\bt\ran} w(|\boxi|)\, |d\bx|.
\end{equation}
Then
\[
\widehat{\phi_\bz}(\bt)=V\bigl(\;\bz+\bt\;\bigr). 
\]
We deduce
\[
\eK^\hbar(\bx,\by)=\sum_{\bnu\in\bZ^m}V\left(\,\bz+\frac{1}{\hbar}\bnu\right),\;\;\bz=\by-\bx.
\]
We set
\begin{equation}
V^\hbar(\bz):=\sum_{\bnu\in\bZ^m} V\left(\,\bz+\frac{1}{\hbar}\bnu\,\right).
\label{eq: vve}
\end{equation}
The function $V^\hbar$ is $(\hbar^{-1}\bZ)^m$-periodic  and moreover,
\begin{equation}\label{eq: cov-asy0}
\eK^\hbar(\bx,\by) =  V^\hbar\bigl(\,\bz\,\bigr),\;\;\bz:=\by-\bx.
\end{equation}
The region $\whB_{1/\hbar}=[-(2\hbar)^{-1},(2\hbar)^{-1}]^m\subset \bR^m$  is a fundamental domain for the action of the lattice  $(\hbar^{-1}\bZ)^m$ on $\bR^m$. From the special  form (\ref{eq: vve}) of $V^\hbar$ and the  fact that $V$ is a Schwartz function we deduce  that for any positive integers $k, N$  we have
 \begin{equation}
 \Vert V^\hbar-V\Vert_{C^k(\whB_{1/\hbar})}=O(\hbar^N)\;\;\mbox{as $\hbar\searrow 0$}.
 \label{eq: vve_appr}
 \end{equation}

From (\ref{covyve})  we deduce  that  $Y^\hbar$ is a stationary Gaussian random  function on $\bR^m$ and its  spectral measure is
\[
\mu^\hbar(d\boxi) =\frac{1}{(2\pi)^m}\sum_{\bnu\in (2\pi\hbar\bZ)^m}(2\pi\hbar)^mw(\bnu)\delta_{\bnu},
\]
where $\delta_{\bnu}$ denotes the Dirac measure on $\bR^m$ concentrated at  $\bnu$.

Let us observe that, as $\hbar \to 0$, the measures $\mu^\hbar(|d\boxi|)$ converge to the  measure 
\[
 \mu^0(|d\boxi|):=\frac{1}{(2\pi)^m}w\bigl(\,|\boxi|\,\bigr)\, |d \boxi|
 \]
in the following sense: for any Schwartz function $u:\bR^m\to\bR$ we have
\[
\int_{\bR^m} u(\boxi)\mu^\hbar(|d\boxi|)= \int_{\bR^m} u(\boxi)\mu^0(|d\boxi|).
\]
  Denote by $Y^0$ the  stationary, Gaussian random function on $\bR^m$ with spectral measure $\mu^0(|d\boxi|)$. Its covariance  kernel is   
  \[
 \eK^0(\bx,\by)= \frac{1}{(2\pi)^m}\int_{\bR^m} e^{-\ii \lan \boxi,\by-\bx\ran} w\bigl(\,|\boxi|\,\bigr)\, |d\boxi|=V(\by-\bx).
 \]
  From (\ref{eq: cov-asy0}) and (\ref{eq: vve_appr}) we deduce that
 \[
 \eK^\hbar\to  \eK^0\;\;\mbox{in $C^\infty$ as $\hbar\to 0$}
 \]
 This suggests that  the statistics of $Y^\hbar$  ought to be ``close'' to the statistics  of $Y^0$.
 
 It is convenient  to give a white noise description of these  random functions.   Recall (see \cite[Chap.7]{Jan}) that a Gaussian white-noise on  $\bR^m$ is a random  measure $W(-)$ that associates to each Borel set $A\in\eB(\bR^m)$ a centered Gaussian random variable $W(A)$ with the property that
 \[
 \bE\bigl[ \, W(A)W(B)\,\bigr]= |A\cap B|.
 \]
 The fact that $Z(-)$  is a random \emph{measure}    is equivalent in this case to the condition
 \[
 W(A\cup B)= W(A)+ W(B),\;\;\forall A, B\in\eB(\bR^m),\;\;A\cap B=\emptyset.
 \]
 Equivalently, a Gaussian white-noise on $\bR^m$  is characterized by a probability space $(\Omega,\eF,\bP)$   and  an isometry 
 \[
 \bsI=\bsI_W: L^2(\bR^m,d\boxi)\to L^2(\Omega,\eF,\bP)
 \]
 onto a Gaussian  subspace of $L^2(\Omega,\eF,\bP)$. More precisely, for $f\in L^2(\bR^m,d\boxi)$ the  Gaussian  random variable $\bsI_W[f]$  is the Ito integral
 \[
 \bsI_W(f)=\int_{\bR^m} f(\xi) W(d\boxi).
 \]
 The isometry property of the Ito integral reads
 \[
 \bsE\Bigl[\,\bsI_W(f)\bsI_W(g)\,\Bigr]=\int_{\bR^m} f(\boxi) g(\boxi)\,d\boxi.
 \]
 In particular
 \[
 W(A)=\bsI_W[\bone_A],\;\;\forall A\in\eB(\bR^m).
 \]
 The existence of  Gaussian white noises is a well settled fact, \cite{GeVi2}. 
 
 Fix  two independent Gaussian white-noises $W_1, W_2$  on $\bR^m$  defined  on the   the probability space $(\Omega,\eF,\bP)$. Let $\bsI_1$ and respectively $\bsI_2$ their associated Ito integrals,
 \[
 \bsI_1,\bsI_2: L^2(\bR^m,d\boxi)\to L^2(\Omega,\eF,\bP).
 \]
 The independence of the white noises  $\bsW_1$ and $\bsW_2$  is equivalent to the condition
 \[
 \bE\bigl[\, \bsI_1(f)\bsI_2(g)\,\bigr]=0,\;\;\forall f,g \in L^2(\bR^m, d\boxi).
 \]
 This shows that we have a  well defined  isometry
 \begin{equation}\label{ito}
 \bsI:   \underbrace{L^2(\bR^m, d\boxi)\times L^2(\bR^m, d\boxi)}_{\mathfrak{H}}\to L^2(\Omega,\eF,\bP),\;\;\bsI( f_1\oplus f_2)=\bsI_1(f_1) +\bsI_2(f_2).
 \end{equation}
 whose image is a Gaussian Hilbert subspace $\eX\subset   L^2(\Omega,\eF,\bP)$.  The map $\bsI$ describes an isonormal Gaussian process parametrized by $\mfH$.  We will use  $\bsI$ to give alternate descriptions to  the functions $Y^\hbar$, $\hbar\geq 0$.
 
 For each $\blam_0\in \bR^m$  and $r>0$ we denote by $C_r\bigl(\blam^0\bigr)$  the cube\footnote{The astute reader may  have observed that $C_r(\blam_0)=\whB_r(\blam^0)$ and  may wonder why the new notation. The reason for  this   redundancy is that the cubes   $\whB$ and $C$ live in \emph{different} vector spaces, dual to each other. The cube $\whB$ lives in the space with coordinates  $\bx$ and $C$ lives in the dual frequency  space with coordinates $\boxi$.} of  size $r$ centered at $\blam_0$  i.e.,
 \[
 C_r\bigl(\blam_0\bigr)=\Bigl\{ \boxi\in\bR^m;\;\;|\boxi-\blam_0|_\infty\leq \frac{r}{2}\,\Bigr\},
 \]
For $\kk\in \bZ^m$ we set  $C_\kk:=C_1(\kk)$. For each $\bx\in \bR^m$ and $\hbar>0$ we set
 \[
 \tilde{Y}^\hbar(\bx):=\sum_{\kk\in\bZ^m} \int_{\bR^m}\sqrt{w(2\pi\hbar |\kk|)}\cos 2\pi\hbar\lan\kk,\bx\ran \bone_{\hbar C_\kk}(\boxi) W_1(d\boxi) 
 \]
 \[
 + \sum_{\kk\in\bZ^m} \int_{\bR^m}\sqrt{w(2\pi\hbar |\kk|)}\sin 2\pi\hbar\lan\kk,\bx\ran \bone_{\hbar C_\kk}(\boxi) W_2(d\boxi) \in\eX .
 \]
 The isometry property of $\bsI$   and (\ref{covyve}) show that
 \[
 \bE\bigl[\, \tilde{Y}^\hbar(\bx)\tilde{Y}^\hbar(\by)\,\bigr]=\eK^\hbar(\bx,\by).
 \]
Thus, the random function $\tilde{Y}^\hbar$ is stochastically equivalent to $Y^\hbar$. Next define
 \[
 \tilde{Y}^0(\bx)=\int_{\bR^m}\sqrt{w(2\pi|\xi|)}\cos 2\pi\lan\boxi,\bx\ran W_1(d\boxi)  +\int_{\bR^m}\sqrt{w(2\pi|\xi|)}\sin 2\pi\lan\boxi,\bx\ran W_2(d\boxi) \in\eX.
 \]
Then
\[
 \bE\bigl[\, \tilde{Y}^0(\bx)\tilde{Y}^0(\by)\,\bigr]=\int_{\bR^m}w\bigl(\, 2\pi|\boxi|\,\bigr) \cos 2\pi\lan\xi,\by-\bx\ran d\boxi
 \]
 \[
 =\int_{\bR^m} e^{-2\pi\ii \lan \boxi,\by-\bx\ran} w\bigl(\,2\pi|\boxi|\,\bigr)\, |d\boxi|=\eK^0(\bx,\by).
 \]
 Thus, the random function $\tilde{Y}^0$ is stochastically equivalent to $Y^0$.  
 
 The above  discussion  shows   that  we can assume that the  Gaussian random variables $Y^\hbar(\bx)$, $\bx\in\bR$, $\hbar\geq 0$, live inside the \emph{same} Gaussian Hilbert space ${\eX}$.
 
 We denote by $\widehat{\eF}$ the  $\si$-subalgebra of $\eF$  generated by the random variables $\bsI(f_1\oplus f_2)$, $f_1\oplus f_2\in\mathfrak{H}$ and we denote by $\widehat{\eX}$ the \emph{Wiener chaos},  \cite{Jan, Maj},
 \begin{equation}\label{whX}
 \widehat{\eX}: =L^2\bigl(\Omega,\widehat{\eF},\bP\,\bigr).
 \end{equation}

\subsection{Statements of the main results} In the sequel we will use the notation $Q_\hbar=O(\hbar^{\infty})$ to indicate that,  for any $N\in\bN$, there exists a constant $C_N>0$ such that 
\[
|Q_\hbar|\leq  C_N\hbar^N\;\;\mbox{as $\hbar\searrow 0$}.
\]
\begin{theorem}\label{th: main1}  Fix a function $N:(0,\infty)\to\bN$ , $\hbar\mapsto N_\hbar$ such that 
\[
 2 N_\hbar   \leq \frac{1}{\hbar},\;\;\forall \hbar>0.
\tag{$\dag$}
\label{dag}
\]
Then, for any box $B\subset \bR^m$ we have
\begin{subequations}
\begin{equation}\label{main1}
\bE\bigl[\,\bsZ(Y^0, B)\,\bigr]=\bar{Z}_0\,|B|,\;\;\bar{Z}_0=\frac{1}{\sqrt{\det(-2\pi\nabla^2 V(0))}} \,\bE\bigl[\,|\det \nabla^2 Y^0(0)|\,\bigr],
\end{equation}
\begin{equation}\label{main2}
\bE\bigl[\,\bsZ(Y^\hbar, \whB_{2N_\hbar})\,\bigr]= \bE\bigl[\,\bsZ(Y^0, \whB_{2N_\hbar})\,\bigr]+O(\hbar^\infty)
\end{equation}
\end{subequations}\qed
\end{theorem}

For simplicity, for any Borel subset $B\subset\bR^m$,  and any $\hbar\in[0,\hbar_0]$ we set
\[
\bsZ^\hbar(B):=\bsZ^\hbar(Y^\hbar, B),\;\;\zeta^\hbar(B)=|B|^{-1/2}\Bigl(\,\bsZ^\hbar(B)-\bE\bigl[\,\bsZ^\hbar(B)\,\bigr]\,\Bigr).
\]
For  $R>0$  we set
\begin{equation}\label{nota}
\bsZ^\hbar(R):=\bsZ^\hbar(\whB_R),\;\;\zeta^\hbar(R)=\zeta^\hbar(\whB_R).
\end{equation}
\begin{theorem}\label{th: main2} There exists  a  number  $S^0>0$ such that,  for any function  
\[
N:(0,\infty)\to\bN, \;\;\hbar\mapsto N_\hbar,
\]
  satisfying  
\[
 2 N_\hbar \leq \frac{1}{2\hbar},\;\;\forall \hbar>0,
\tag{$\ddag$}
\label{ddag}
\]
and
\[
\lim_{\hbar\to 0}  N_\hbar =\infty,
\tag{$\ast$}
\label{ast}
\]
the following hold.

\begin{enumerate}

\item As $\hbar\to 0$  
\[
\var\bigl[\, \bsZ^\hbar(2N_\hbar)\,\bigr]\sim S^0 \cdot (2N_\hbar)^m\sim \var\Bigl[ \bsZ^0 (2N_\hbar)\Bigr].
\]

\item The families of random variables
\[
\Bigl\{\,\zeta^\hbar(2N_\hbar)\,\Bigr\}_{\hbar\in(0,\hbar_0]}\;\;\mbox{and}\;\;\Bigl\{\, \zeta^0(2N_\hbar)\,\Bigr\}_{\hbar\in(0,\hbar_0]}
\]
converge in distribution  as $\hbar\to 0$ to normal random variables $\sim \eN(0, S^0)$.
\end{enumerate}
\end{theorem}

Recall that  for any $r\in (0,1]$  we have denoted by $B_r\subset \bT^m$  the image of $\whB_r$ under the natural projection $\bR^m\to\bT^m$.   Note that $B_1=\bT^m$ and recall (\ref{z_resc}),
\[
\bsZ(Y^\hbar, \whB_{\hbar^{-1}r})=\bsZ(X^\hbar, B_r),\;\;\forall r\in (0,1].
\]
If $r$ is  fixed in $(0,1/2]$,  but  $\hbar\to 0$ in such  a way that 
\[
\frac{r}{2\hbar}=N_\hbar\in\bN
\tag{$Q$}
\label{Q}
\] 
 then  Theorem \ref{th: main2}  provides an information on the behavior  on $\bsZ(X^\hbar, B_r)$ as $\hbar\to 0$.   Our next    shows that we reach the same conclusion without  assuming the quantization condition (\ref{Q}).
 
\begin{theorem}\label{th: main3} Let $\bar{Z}_0$ be as in (\ref{main1}) and  $S_0$ be as in Theorem \ref{th: main2}.  Assume\footnote{The assumption  $r\leq 1/2$ is the counterpart of (\ref{ddag}).} that  
\[
0<r\leq \frac{1}{2}.
\]
 Then   the following hold.

\begin{enumerate}

\item  As $\hbar\to 0$,
\[
\bE\bigl[\, \bsZ(X^\hbar, B_r)\,\bigr] = \hbar^{-m}\bigl(\,  \bar{Z}_0 \vol(B_r)+O(\hbar^\infty)\,\bigr).
\]

\item As $\hbar\to 0$ we have
\[
\var\bigl[\, \bsZ(X^\hbar, B_r)\,\bigr]\sim S_0 \hbar^{-m} \vol(B_r).
\]

\item As $\hbar\to 0$ the random variables 
\begin{equation}\label{z_tor}
\left(\frac{\hbar}{r}\right)^{\frac{m}{2}} \Bigl(\,\bsZ(X^\hbar, B_r)-\bE\bigl[\, \bsZ(X^\hbar, B_r)\,\bigr]\,\Bigr)
\end{equation}
converge in distribution to a random variable $\sim \eN(0, S^0)$.
\end{enumerate} 
\end{theorem}

\begin{corollary}\label{cor: main1} Let $r\in(0,1/2)$. Set
\[
\bar{\bsZ}^\hbar(X^\hbar, B_r):= \hbar^m  {\bsZ}^\hbar(X^\hbar, B_r).
\]
Let $(\hbar_n]$ be a  sequence of postive numbers such that, for some $p\in (0,m)$ we have
\[
\sum_{n\geq 1} \hbar_n^p<\infty.
\]
Then
\[
\bar{\bsZ}^{\hbar_n}(X^\hbar, B_r)\to  \bar{Z}_0 \vol(B_r) \;\;a.s..
\]
\end{corollary}
\begin{proof} Set $\alpha:=\frac{m-p}{2}$ so that $p= m-2\alpha$.  Note that
\[
\bE\bigl[ \, \bar{\bsZ}^{\hbar_n}(X^\hbar, B_r)\,\bigr]= \bar{Z}_0 \vol(B_r)+O(\hbar^\infty)
\]
\[
\bigl[ \, \bar{\bsZ}^{\hbar_n}(X^\hbar, B_r)\,\bigr]\sim S_0\vol(B_r)\hbar^m.
\]
From Chebyshev's  inequality  we deduce
\[
\bP\bigl[\,  \bigl|\bar{\bsZ}^{\hbar_n}(X^\hbar, B_r)- \bar{Z}_0 \vol(B_r)\bigr|\geq h_n^\alpha\bigr] = O(\hbar_n^p\bigl).
\]
Then 
\[
\sum_{n\geq 1}\bP\bigl[\,  \bigl|\bar{\bsZ}^{\hbar_n}(X^\hbar, B_r)- \bar{Z}_0 \vol(B_r)\bigr|\geq h_n^\alpha\bigr]<\infty,
\]
and the first Borel-Cantelli lemma    implies the desired conclusion.
\end{proof} 

Our final result extends Theorem \ref{th: main3} to the degenerate case $r=1$ when the assumption (\ref{ddag}) is not satisfied.

\begin{theorem}\label{th: main4} Let $\bar{Z}_0$ be as in (\ref{main1}) and  $S_0$ be as in Theorem \ref{th: main2}.  Then   the following hold.

\begin{enumerate}

\item  As $\hbar\to 0$,
\[
\bE\bigl[\, \bsZ(X^\hbar, \bT^m)\,\bigr] = \hbar^{-m}\bigl(\,  \bar{Z}_0+O(\hbar^\infty)\,\bigr).
\]

\item As $\hbar\to 0$ we have
\[
\var\bigl[\, \bsZ(X^\hbar, \bT^m)\,\bigr]\sim S_0 \hbar^{-m}.
\]

\item As $\hbar\to 0$ the random variables 
\begin{equation}\label{z_tor1}
\hbar^{\frac{m}{2}} \Bigl(\,\bsZ(X^\hbar, \bT^m)-\bE\bigl[\, \bsZ(X^\hbar, \bT^m)\,\bigr]\,\Bigr)
\end{equation}
converge in distribution to a random variable $\sim \eN(0, S^0)$.

\end{enumerate}

\end{theorem}

\begin{remark} (a) In principle, the method of proof  we employ can produce effective  bounds on the  Fortet-Mourier  on the space of probability  measures on the real axis.     For the sake of  clarity we have have decided not to pursue this aspect.

(b) The constant  $\bar{Z}_0$ in  (\ref{main1}) depends only on $m$ and $w$, $\bar{Z}_0=\bar{Z}_0(w,m)$ an represents the  expected density of  critical points per  unit volume of the random function $Y^0$. We set
 \begin{equation}
 I_k(w):=\int_0^\infty w(r) r^k dr.
 \label{eq: Ik}
 \end{equation}
We have (see \cite{Nispec})
 \begin{equation}\label{sdh}
  \begin{split}
 (2\pi)^{m/2} s_m=\frac{2\pi^{\frac{m}{2}}}{ \Gamma(\frac{m}{2})} I_{m-1}(w),\;\;(2\pi)^{m/2} d_m= \frac{2\pi^{\frac{m}{2}}}{m \Gamma(\frac{m}{2})}I_{m+1}(w),\\
  (2\pi)^{m/2} h_m=\frac{1}{3}\int_{\bR^m}x_1^4 w(|x|) dx=\frac{2\pi^{\frac{m}{2}}}{m(m+2) \Gamma(\frac{m}{2})}I_{m+3}(w).
  \end{split}
  \end{equation}
Denote by $\eS_m$    the  Gaussian Orthogonal  Ensemble  of  real symmetric $m\times m$ matrices   $A$  with  independent,  normally distributed  entries   $(a_{ij})_{1\leq i,j\leq m}$ with variances
\[
\bsE[a_{ii}^2]=2,\;\;\bsE[a_{ij}^2]=1,\;\;\forall 1\leq i\neq j\leq m
\]  
As explained in \cite{Niclt}, we have
  \begin{equation}\label{cmw}
\bar{Z}_0(w,m)= \left(\frac{h_m}{2\pi d_m}\right)^{\frac{m}{2}}\bE_{\eS_m}\bigl[\, |\det A|\,\bigr]=\left(\frac{I_{m+1}(w)}{2\pi (m+2) I_{m+3}(w)} \right)^{\frac{m}{2}}\bE_{\eS_m}\bigl[\,|\det A|\,\bigr].
\end{equation}
In \cite[Cor.1.7]{Nispec} we have shown that, as $m\to \infty$, we have
\begin{equation}
 \bar{Z}_0(w,m)\sim \frac{8}{\sqrt{\pi m}}\Gamma\left(\frac{m+3}{2}\right)\left(\frac{2I_{m+3}(w)}{\pi (m+2) I_{m+1}(w)}\right)^{\frac{m}{2}}.
 \label{eq: cw}
 \end{equation}
The asymptotic behavior  of $\bar{Z}_0(w,m)$ as $m\to\infty$   depends  rather dramatically on the size of the tail of the Schwartz function $w$: the heavier the tail, the  faster the growth of  $\bar{Z}_0(w,m)$ as $m\to\infty$.  For example, in \cite[Sec.3]{Nispec} we have shown  the following.

\begin{itemize}

\item  If $w(t)\sim  \exp(-\log t)\log(\log t)\,)$ as $t\to \infty$, then
\[
\log \bar{Z_0}(w,m)\sim \frac{m}{2}e^{m+2}(e^2-1)\;\;\mbox{as $m\to \infty$}.
\]
\item If  $w(t)\sim \exp\bigl(-(\log t)^{\frac{p}{p-1}}\;\bigr)$ as $t\to\infty$, $p>1$, then, for some explicit constant $C_p>0$, we have
\[
\log \bar{Z}_0(w,m)\sim C_p m^{p},\;\; \mbox{as $m\to \infty$}
\]
\item If $w(t)\sim e^{-t^2}$ as $t\to \infty$, then
\[
 \log \bar{Z}_0(w,m)\sim \sim \frac{m}{2}\log m, \;\; \mbox{as $m\to \infty$}.
 \]
 \end{itemize}
 
 \smallskip
 
 \noindent (c) The constant  $S_0$ in Theorem \ref{th: main2} seems very difficult to estimate.   As the proof of Theorem \ref{th: main2} will show, the constant $S_0$    is a sum of  a series with nonnegative terms
 \[
 S_0=\sum_{q\geq 1}\bar{S}^0_q,
 \]
 where the terms  $\bar{S}^0_q$ are defined explicitly in (\ref{rnlim}). In \cite{Niclt} we have proved that $S_0>0$ by showing  that 
 \[
 \bar{S}^0_2=\int_{\bR^m} \bigl|\,P(\xi_1,\dotsc,\xi_m) w(|\boxi|)\,\bigr|^2 d\boxi,
 \]
 where $P(\xi_1,\dotsc,\xi_m)$ is a certain  explicit, nonzero,  but rather complicated polynomial. The constant $ \bar{S}^0_2$ depends on  $w$ and $m$. In \cite[Appendix A]{Niclt} we described   methods of producing asymptotic estimates for $\bar{S}^0_2(w,m)$ as $m\to\infty$, but the results are not too pretty. \qed
\end{remark}

\subsection{Outline of proofs} \label{ss: 14}  The strategy of proof  is inspired from \cite{AzLe, EL}.  As explained  earlier, the  Gaussian random variables  $Y^\hbar(\bx)$, $\bx\in\bR^m$, $\hbar\geq 0$,  are defined on  the same probability space $(\Omega, \eF, \bsP)$ and inside the same Gaussian Hilbert space $\eX$.  

Using the Kac-Rice formula  and the asymptotic estimates in \cite{Nifluct} we  show in  Subsection \ref{ss: 21}  that, for any $\hbar\geq 0$ sufficiently small, and any box $B$, the random variables $\bsZ^\hbar(B)$  belongs to the Wiener chaos $\widehat{\eX}$  and we  describe  its  Wiener chaos decomposition.  The key result  behind this fact is Proposition \ref{prop: key} whose rather involved technical proof  is deferred to Appendix \ref{s: key}.   The Wiener chaos decomposition of   $\bsZ^\hbar(B)$  leads  immediately  to (\ref{main1}) and (\ref{main2}).

To prove that the random variables $\zeta^\hbar(2N_\hbar)$ and $\zeta^0(2N_\hbar)$ converge in law to  normal  random variable $\bar{\zeta}^0(\infty)$ and respectively $\zeta^0(\infty)$  we imitate the strategy in \cite{EL, Niclt} based   on the    very general Breuer-Major type  central limit  theorem \cite[Thm. 6.3.1]{NP}, \cite{NP2009, NPP, NuaPe, PeTu}.      

The  case  of the variables  $\zeta^0(N_\hbar)$  is  covered in \cite{Niclt} where we have shown that there exists $S^0>0$ and a normal random bariable $\zeta^0(\infty)\sim\eN(0,S^0)$  such that, as $N\to\infty$, the random variable $\zeta^0(N)$ converges in law to $\zeta^0(\infty)$.  

The case    $\zeta^\hbar(2N_\hbar)$  is conceptually similar, but the  extra dependence on $\hbar$ adds an extra layer of difficulty.  Here are the details.  

Denote by $\zeta^\hbar_q$ the $q$-th chaos component   of $\zeta^\hbar_q(2N_\hbar)\in\widehat{\eX}$.   According to \cite[Thm.6.3.1]{NP},  to prove  that $\zeta^\hbar(2N_\hbar)$ converges in law to a  normal random variable  $\bar{\zeta}^0(\infty)$ it suffices to prove the following.

\begin{enumerate}

\item  For every $q\in\bN$ there exists $\bar{S}^0_{q}\geq 0$ such that
\[
\lim_{\hbar\to 0} \var\bigl[\,\zeta^\hbar_{q}\bigr]=\bar{S}^0_{q}.
\]
\item  Exists $\hbar_0>0$ such that
\[
\lim_{Q\to\infty} \sup_{0\leq \hbar\leq \hbar_0} \sum_{q\geq Q} \var\bigl[\,\zeta^\hbar_{q} \,\bigr]=0.
\]
\item  For   each $q\in\bN$, the random variables $\zeta^\hbar_{q}(2N_\hbar)$ converge in law   to a normal random variable, necessarily of variance $\bar{S}^0_{q}$.
\end{enumerate}

We prove (i) and (ii) in   Subsection \ref{ss: 24}; see  (\ref{vq2}) and respectively Lemma \ref{lemma: VQ}.

To prove   (iii) we rely on the fourth-moment theorem \cite[Thm. 5.2.7]{NP}, \cite{NuaPe}. The details are identical to the ones employed in the proof of  \cite[Prop. 2.4]{EL}. The variance  of the  limiting normal random variable $\bar{\zeta}^0(\infty)$ is
\[
\var\bigl[\,\bar{\zeta}_0(\infty)\,\bigr]=\sum_{q\geq 1} \bar{S}^0_{q}  <\infty.
\]
The explicit description of the components $\bar{S}^0_q$   will then show that $S^0=\bar{S}^0$.

The proof of Theorem \ref{th: main3} is, up to a suitable rescaling, identical to the proof of Theorem \ref{th: main2}. We explain this  in more detail in Subsection \ref{ss: cor}. The proof of Theorem \ref{th: main4} requires a  clever modification of the arguments in  the proof of Theorem \ref{th: main3}   because  in this case  the condition(\ref{ddag}) is not satisfied. The details are contained in Subsection \ref{ss: main4}

\subsection{Related results} Central limit  theorems concerning  crossing counts of random functions go back a while, e.g. T. Malevich \cite{Mal69} (1969) and J. Cuzik \cite{Cuz76} (1976).

The usage of  Wiener chaos decompositions  and of  Breuer-Major type results  in proving  such  central limit  theorems is more recent, late 80s early 90s. We want to mention here the pioneering  contributions of  Chambers and Slud \cite{ChaSl},   Slud \cite{Slud91, Slud94},  Kratz and Le\'{o}n \cite{KL1997},  Sodin and Tsirelson \cite{ST}.

 This topic  was further elaborated by  Kratz  and Le\'{o}n  in \cite{KL2001}  where they  also proved a central limit theorem  concerning the length of the zero set of a random function of two variables. We refer to \cite{AzWs}  for  particularly nice discussion of these  developments.   
  
  Aza\"{i}s and Le\'{o}n \cite{AzLe} used  the technique of  Wiener  chaos decomposition to give a  shorter and more conceptual proof   to a central limit theorem   due to Granville and Wigman \cite{GW}  concerning the number  of zeros of random trigonometric polynomials of large degree. This technique was then successfully   used by  Aza\"{i}s, Dalmao and Le\'{o}n \cite{AzDaLe}  to prove a CLT concerning the number of zeros of  Gaussian   even trigonometric polynomials and by Dalmao  in \cite{Da} to  prove a CLT concerning the number of zeros of   one-variable  polynomials   in the Kostlan-Shub-Smale probabilistic ensemble.  Adler and Naitzat \cite{AN} used  Hermite  decompositions  to  prove a CLT  concerning Euler integrals of random functions.
  
  The recent results \cite{ADLNP} suggests that the central  limit results proved in  this paper may have a universal character in the sense  that  the random Fourier series (\ref{uve}) need not be Gaussian.

\section{Proofs of the main results}
\setcounter{equation}{0}

\subsection{Hermite decomposition of the number of critical points}\label{ss: 21}

For every $\hbar\geq 0$, $\bv\in\bR^m$ and $B\in\eB(\bR^m)$ we denote by $\bsZ^\hbar(\bv, B)$  the number of solutions $\bx$ of the equation
\[
\nabla Y^\hbar(\bx)=\bv,\;\;\bx\in B.
\]
For $\ve>0$ we define
\[
\delta_\ve:\bR^m\to\bR,\;\; \delta_\ve(\bv)= \ve^{-m} \bone_{\whB_{\ve/2}(0)}(\bv).
\]
Note that $\delta_\ve$ is supported  on the cube of size $\ve$ centered at the origin and its total integral is $1$. As $\ve\searrow 0$, the function $\delta_\ve$ converges in the sense of distributions to  the Dirac $\delta_0$.  We set
\[
\bsZ^\hbar_\ve(\bv, B)= \int_B \bigl|\,\det \nabla^2 Y^\hbar(\bx)\bigr|\;\delta_\ve\bigl(\,\nabla Y^\hbar(\bx)-\bv\,\bigr)d\bx,\;\;\bsZ^\hbar(B):=\bsZ^\hbar(\bv, B)\bigr|_{\bv=0}.
\]
We define a \emph{box}   in $\bR^m$  to be a set $B\subset\bR^m$ of the form
\[
B=[a_1,b_1]\times \cdots\times [a_m,b_m],\;\;a_1<b_1,\dotsc, a_m<b_m.
\]
If $B\subset \bR^m$ is a box   \cite[Thm.11.3.1]{AT}, we deduce that $X$  is a.s. a Morse function on $T$ and in particular, for any $\bv\in\bR^m$, the equation $\nabla X^\hbar(\bx)=\bv$  almost surely has no solutions $\bx\in\pa B$.

The proof of the Kac-Rice formula \cite[Thm. 11.2.3]{AT} shows  that $\bsZ^\hbar(\bv, B)\in L^1(\Omega)$ and  
\[
\bsZ^\hbar_\ve(\bv, B)\to \bsZ^\hbar(\bv, B)\;\;\mbox{a.s. as $\ve \to 0$}.
\]

\begin{proposition}\label{prop: key} There exists $\hbar_0>0$, sufficiently small,   and $C_0>0$ such that, for any $\hbar\in [0,\hbar_0)$ and any box $B\subset \whB_{2}(0)$ the following hold.   

\begin{enumerate} 

\item For any $\bv\in\bR^m$, $\bsZ^\hbar(\bv, B)\in L^2(\Omega,\widehat{\eF},\bP)$.

\item  The function
\[
\bR^m\ni \bv\mapsto \bE\bigl[\, \bsZ^\hbar(\bv, B)^2\,\bigr]\in\bR
\]
is continuous.

\item  For any $\bv\in\bR^m$
\[
\lim_{\ve\to 0} \bsZ^\hbar_\ve(\bv, B)= \bsZ^\hbar(\bv, B)\;\;\mbox{in $L^2(\Omega)$}.
\]

\item The function
\[
[0,\hbar_0]\ni\hbar\mapsto \bsZ^\hbar(B)\in L^2(\Omega,\widehat{\eF},\bP)
\] 
is continuous.
\end{enumerate}
\end{proposition}

We defer  the proof of Proposition \ref{prop: key} to the Appendix \ref{s: key}.  The case $\hbar=0$ of this proposition is discussed in \cite[Prop.1.1]{EL}. That proof    uses in an essential fashion the isotropy of the  random function $Y^0$.  The random functions  $Y^\hbar$, $\hbar\neq 0$, are not  isotropic,  but  they are ``nearly'' so for $\hbar$ small.  

Since for any Borel set $B\subset \bR^m$, and any $\ve>0$  the  random variables $\bsZ_\ve^\hbar(\bv, B)$ belong to the Wiener chaos $\widehat{\eX}$  defined in (\ref{whX}), we deduce from  Proposition \ref{prop: key}(iii)   that,  for  any $\hbar\leq \hbar_0$, and any  box $B\subset \whB_{1/\hbar}$,  the  number of critical points $\bsZ^\hbar(B)$ belongs to the Wiener chaos $\widehat{\eX}$.

Fix $\hbar_0$ as in Proposition \ref{prop: key}. Consider the random  field
\[
\tbY^\hbar(\bx):=\nabla Y^\hbar(\bx) \oplus  \nabla^2 Y^\hbar(\bx),\;\;\bx\in\bR^m,\;\;\hbar\in[0,\hbar_0].
\]
of dimension
\[
D=m+\nu(m),\;\;\nu(m):=\frac{m(m+2)}{2}.
\]
Note that 
\[
\bE\bigl[\, Y^\hbar_i(\bx) Y^\hbar_{j,k}(\bx)\,\bigl]= -V^\hbar_{i,j,k}(0)=0,
\]
since $V^\hbar(\bx)$ is an even  function.  Hence, the two components of $\tbY^\hbar$ are independent.   We  can  find   invertible matrices $\Lambda_1^\hbar$ and  $\Lambda^\hbar_2$ of dimensions $m\times m$ and respectively $\nu(m)\times \nu(m)$, that depend continuously on $\hbar\in [0,\hbar_0]$ such that the probability distributions of the random vectors 
\[
U^\hbar(\bx)= (\Lambda_1^\hbar)^{-1}\nabla Y^\hbar(\bx)\in\bR^m,\;\;A^\hbar(\bx):=(\Lambda_2^\hbar)^{-1} \nabla^2 Y^\hbar(\bx)\in\bR^{\nu(m)}
\]
are the canonical Gaussian measures on the Euclidean spaces $\bR^m$ and $\bR^{\nu(m)}$ respectively.   More precisely,   we can choose as $\Lambda_i^\hbar$, $i=1,1$, the square roots of the covariance matrices of $\nabla^i Y^\hbar(\bx)$.

Consider the functions
\[
f^\hbar:\bR^{\nu(m)}  \to\bR,\;\; f^\hbar(A)=\bigl|\det \Lambda_2^\hbar A\,\bigr|,
\]
\[
G^\hbar_\ve: \bR^m\times \bR^{\nu(m)}\to\bR,\;\; G_\ve(U,A)= \delta_\ve(\Lambda_1^\hbar U) f_\hbar(A).
\]
\emph{Fix a box $B$, independent of $\hbar$}.   Proposition \ref{prop: key} shows that, for $\hbar$ sufficiently small, we have
\[
\bsZ^\hbar(B)= \lim_{\ve \to 0} \int_B G^\hbar_\ve(U(\bx), A(\bx) ).
\]
Recall that an orthogonal basis of $L^2\bigl(\, \bR,\bgamma(dx)\,\bigr)$ is given   by  the Hermite polynomials, \cite[Ex. 3.18]{Jan},  \cite[V.1.3]{Mal},
\begin{equation}\label{hermn}
H_n(x) :=(-1)^n e^{\frac{x^2}{2}} \frac{d^n}{dx^n}\Bigl(\, e^{-\frac{x^2}{2}}\,\Bigr)= n!\sum_{r=0}^{\lfloor\frac{n}{2}\rfloor} \frac{(-1)^r}{2^r r!(n-2r)!} x^{n-2r}.
\end{equation}
In particular
\begin{equation}\label{hn0}
H_n(0)=\begin{cases}
0, & n\equiv 1\bmod 2,\\
(-1)^r\frac{(2r)!}{2^rr!}, & n=2r.
\end{cases}
\end{equation}
For every multi-index  $\alpha=(\alpha_1, \alpha_2,\dotsc )\in\bN_0^{\bN}$ such that    all but finitely many $\alpha_k$-s are nonzero, and any
\[
\underline{x}=(x_1,x_2,\dotsc)\in\bR^{\bN}
\]
 we set
\[
|\alpha|:=\sum_k\alpha_k,\;\;\alpha!:=\prod_k \alpha_k!,\;\;H_\alpha(\underline{x}):=\prod_k H_{\alpha_k}(x_k).
\]
Following \cite[Eq.(5)]{EL}   we define for every $\alpha\in\bN_0^m$ the quantity
\begin{equation}\label{dalph}
d_\alpha:=\frac{1}{\alpha!}(2\pi)^{-\frac{m}{2}}  H_\alpha(0).
\end{equation}
The function $f^\hbar: \bR^{\bnu(m)}\to\bR$   has a  $L^2(\bR^{\nu(m)},\bGamma)$-orthogonal decomposition
\[
f^\hbar(A)= \sum_{n\geq 0} f_n^\hbar(A),
\]
where 
\begin{equation}\label{fn}
f_n^\hbar(A)=\sum_{\substack{\beta\in\bN_0^{\nu(m)},\\ |\beta|=n}} f_\beta^\hbar H_\beta(A),\;\;f_\beta^\hbar=\frac{1}{\beta!}\int_{\bR^{\nu(m)}} f^\hbar(A)  H_\beta(A) \bGamma(dA).
\end{equation}
Note that
\begin{equation}\label{f0}
f_0^\hbar=\bE\bigl[\,|\det \nabla^2 Y^\hbar(0)|\,\bigr].
\end{equation}
The function $\delta_\ve(U)$ has an $L^2(\bR^m, \bGamma)$-orthogonal decomposition
\[
\delta_\ve(U)= \sum_{\alpha\in\bN_0^m}  d^\hbar_{\alpha,\ve} H_\alpha(U),
\]
where
\[
d^\hbar_{\alpha,\ve}=\frac{1}{\alpha!}\int_{\bR} \delta_\ve(\Lambda_1^\hbar U) H_\alpha(U) \bGamma( dU).
\]
Note that
\[
\lim_{\ve\to 0} \int_{\bR} \delta_\ve(\Lambda_1^\hbar U) H_\alpha(U) \bGamma( dU) = \frac{1}{\det \Lambda_1^\hbar}  H_\alpha(0),
\]
so that
\begin{equation}\label{dhve}
\lim_{\ve\to 0} d^\hbar_{\alpha,\ve}=  \frac{1}{\det \Lambda_1^\hbar} d_\alpha,
\end{equation}
uniformly  for $\hbar\in [0,\hbar_0]$. We set
\[
\omega_\hbar:= \frac{1}{\det \Lambda_1^\hbar}.
\]
\begin{remark}\label{rem: lam}  The  matrix $\Lambda^\hbar_1$ is the square root of the covariance matrix  of the random vector $\nabla Y^\hbar(0)$, i.e.,
\[
\Lambda_1^\hbar=\sqrt{-\nabla^2 V^\hbar(0)}.
\]
The function $V=V^{\hbar=0}$ is radially symmetric  and thus
\[
\nabla^2 V(0)= -\lambda^2 \one_m,
\]
for some $\lambda>0$. Hence
\begin{equation}\label{lam}
\Lambda_1^0=\lambda\one_m,\;\;\omega_0=\lim_{\hbar\to 0} \omega_\hbar=\lambda^{-m}=\frac{1}{ \sqrt{\det (-\nabla^2V(0))}}.\end{equation}
\qed
\end{remark}
If we set
\[
\eI_m:=\bN_0^m\times \bN_0^{\nu(m)}
\]
Then
\begin{equation}\label{zqve}
\bsZ^\hbar_\ve(B)=\sum_{q=0}^\infty \int_B \rho^\hbar_{q,\ve}(\bx ) d\bx,
\end{equation}
where 
\[
\rho^\hbar_{q,\ve}(\bx )=\sum_{\substack{(\alpha,\beta)\in\eI_m,\\ |\alpha|+|\beta|=q}} d^\hbar_{\alpha,\ve}f^\hbar_\beta H_\alpha(U(\bx)) H_\beta(A(\bx)).
\]
If we let $\ve\to 0$ in (\ref{zqve}) and  use Proposition \ref{prop: key}(iii) and  (\ref{dhve}) we deduce
\begin{subequations}
\begin{equation}\label{za}
Z^\hbar(B)=\sum_{q\geq 0} Z_q^\hbar(B),\;\;Z^\hbar_q(B)=\int_B\rho^\hbar_q(\bx) d bx,
\end{equation}
\begin{equation}\label{zb}
\rho_q^\hbar(\bx) =\sum_{\substack{(\alpha,\beta)\in\eI_m,\\ |\alpha|+|\beta|=q}} \omega_\hbar d_\alpha f^\hbar_\beta H_\alpha(U(\bx)) H_\beta(A(\bx)).
\end{equation}
\end{subequations}
To proceed further  we need to use   some basic Gaussian estimates.

\subsection{A technical interlude}  Let $\bsV$ be a real Euclidean space  of dimension $N$.  We denote by $\eA(\bsV)$ the space of symmetric positive semidefinite operators $A:\bsV\to\bsV$.  For $A\in \eA(\bsV)$ we denote  by $\bgamma_A$ the centered   Gaussian measure on $\bsV$ with  covariance form $A$. Thus
\[
\bgamma_{\one}(d\bv)=\frac{1}{(2\pi)^{\frac{N}{2}}}e^{-\frac{1}{2}|\bv|^2} d\bv,
\]
and $\bgamma_A$ is the push forward of $\bgamma_\one$ via the linear map $\sqrt{A}$,
\begin{equation}
\bgamma_A=(\sqrt{A})_* \bgamma_{\one}.
\label{eq: push}
\end{equation}
For any  measurable $f:\bsV\to\bR$ with at most polynomial growth  we set
\[
\bsE_A(f)=\int_{\bsV} f(\bv) \bgamma_A(d\bv).
\]
\begin{proposition} Let $f:\bsV\to \bR$ be a locally  Lipschitz function which is positively homogeneous of degree $\alpha\geq 1$.   Denote by $L_f$ the  Lipschitz constant of the restriction of  $f$ to the unit ball of $\bsV$.  There exists a constant $C>0$ which depends only on $N$ and $\alpha$  such that, for any $\Lambda>0$  and any $A,B\in \eA(\bsV)$ such that $\Vert A\Vert,\Vert B\Vert\leq \Lambda$ we have
\begin{equation}
\bigl|\, \bsE_A(f)-\bsE_B(f)\,\bigr|\leq  CL_f \Lambda^{\frac{\alpha-1}{2}} \Vert A- B\Vert^{\frac{1}{2}}.
\label{eq: holder}
\end{equation}
\label{prop: diff-gauss}
\end{proposition}

\begin{proof}    We present the very elegant argument we learned from George Lowther on \href{http://mathoverflow.net/questions/130496/continuous-dependence-of-the-expectation-of-a-r-v-on-the-probability-measure}{MathOverflow}. In the sequel we will use the same  letter $C$ to denote various constant that depend only on $\alpha$ and $N$.

First of all let us observe that  (\ref{eq: push}) implies that
\[
\bsE_A(f)=\int_{\bsV} f(\sqrt{A}\bv) \bgamma_{\one}(d\bv).
\]
We deduce that for any $t>0$ we have
\[
\bsE_{tA}(f) =\int_{\bsV} f(\sqrt{tA}\bv) \bgamma_{\one}(d\bv)=t^{\frac{\alpha}{2}}\int_{\bsV} f(\sqrt{A}\bv) \bgamma_{\one}(d\bv)=t^{\frac{\alpha}{2}}\bsE_A(f),
\]
and thus it suffices to prove (\ref{eq: holder}) in the special case $\Lambda=1$, i.e. $\Vert A\Vert,\Vert B\Vert\leq 1$.    We have
\[
\bigl|\, \bsE_A(f)-\bsE_B(f)\,\bigr|\leq \int_{\bsV} \bigl|\, f(\sqrt{A}\bv)-f(\sqrt{B}\bv)\,\bigr| \bgamma_{\one}(d\bv)
\]
\[
=\int_{\bsV}|\bv|^\alpha \Bigl|\, f\Bigl(\sqrt{A}\frac{1}{|\bv|}\bv\Bigr)-f\Bigl(\sqrt{B}\frac{1}{|\bv|}\bv\Bigr)\,\Bigr| \bgamma_{\one}(d\bv)
\]
\[
\leq L_f \int_{\bsV}|\bv|^\alpha \Bigl|\,\sqrt{A}\frac{1}{|\bv|}\bv-\sqrt{B}\frac{1}{|\bv|}\bv\,\Bigr|\bgamma_{\one}(d\bv)
\]
\[
\leq L_f\Vert\sqrt{A}-\sqrt{B}\Vert \int_{\bsV}|\bv|^\alpha \bgamma_{\one}(d\bv)\leq CL_f\Vert A-B\Vert^{\frac{1}{2}}.
\]
\end{proof}

\subsection{Proof of Theorem \ref{th: main1} }

Note that
\[
\bE\bigl[\,\bsZ^\hbar(B)\,\bigr]= \bE\bigl[\, Z_0^\hbar(B)\,\bigr]= |B| \omega_\hbar f^\hbar_0 d(0)
\]
(use (\ref{dalph}) and (\ref{f0}) )
\[
=(2\pi)^{-m/2}|B|\omega_\hbar\bE\bigl[\,|\det \nabla^2 Y^\hbar(0)|\,\bigr].
\]
Using (\ref{eq: cov-asy0}) and Remark \ref{rem: lam} we deduce that that
\[
\omega_\hbar-\omega_0=O(\hbar^\infty).
\]
From  (\ref{eq: cov-asy0}) we also deduce that
\[
\Vert \nabla^2 V^\hbar(0)-\nabla^2 V(0)\Vert =O(\hbar^\infty.
\]
Invoking Proposition \ref{prop: diff-gauss} we deduce that
\[
\bE\bigl[\,|\det \nabla^2 Y^\hbar(0)|\,\bigr]= \bE\bigl[\,|\det \nabla^2 Y^0(0)|\,\bigr] +O(\hbar^\infty).
\]
Hence
\begin{equation}\label{main1a}
\begin{split}
\bE\bigl[\,\bsZ^\hbar(B)\,\bigr]= \bE\bigl[\,\bsZ^0(B)\,\bigr] +O(\hbar^\infty),\\  \bE\bigl[\,\bsZ^0(B)\,\bigr] =(2\pi)^{-m/2}\omega_0\,|B|\,\bE\bigl[\,|\det \nabla^2 Y^0(0)|\,\bigr].
\end{split}
\end{equation}
Using (\ref{lam}) in the above  equality we obtain  (\ref{main1}).

Let $N_\hbar$ satisfy (\ref{dag}). Recall that $\bA^m$ denotes the affine lattice
\begin{equation}\label{ash}
\bA^m=\left(\frac{1}{2}+\bZ\right)^m.
\end{equation}
We have
\begin{equation}\label{ash1}
\whB_{2N_\hbar}=\bigcup_{\ba \in \bA^m,\;|\ba|_\infty\leq N_\hbar} \whB(\ba),\;\;\whB(a):=\whB_1(\ba).
\end{equation}
The cubes in the above union   have disjoint interiors.  According to  \cite[Thm.11.3.1]{AT}, for $\hbar\leq \hbar_0$ the function $Y^\hbar$ is a.s.  Morse. Give a box $B\subset \bR^m$, the function $Y^\hbar$ will a.s. have no critical points on the boundary of $B$  Thus
\[
\bsZ^\hbar\bigl(\, \whB_{2N_\hbar}\,\bigr) =\sum_{\ba \in \bA^m\cap\whB_{2N_\hbar}}\bsZ^\hbar\bigl(\, \whB(\ba)\,\bigr).
\]
From  (\ref{dag}) we deduce that $\whB(\ba)\subset \whB_{1/\hbar}(0)$ so  (\ref{eq: vve_appr}) holds on $\whB(\ba)$. We deduce
\[
\bsZ^\hbar\bigl(\, \whB_{2N_\hbar}\,\bigr)\stackrel{(\ref{main1a})}{=}\sum_{ \ba \in \bA^m\cap\whB_{2N_\hbar}}\Bigl(\,\bsZ^0\bigl(\, \whB(\ba)\,\bigr) + O(\hbar^\infty)\,\Bigr)
\]
From (\ref{dag}) we deduce  that $2N_\hbar\leq \frac{1}{\hbar}$ so that
\[
\#\bigl(\ba \in \bA^m\cap\whB_{2N_\hbar}\,\bigr) =O(\hbar^{-m}).
\]
Hence 
\[
\bsZ^\hbar\bigl(\, \whB_{2N_\hbar}\,\bigr)=\left(\sum_{\ba \in \bA^m\cap\whB_{2N_\hbar}}\bsZ^0\bigl(\, \whB(\ba)\,\bigr)\,\right) + O(\hbar^\infty)=\bsZ^\hbar\bigl(\, \whB_{2N_\hbar}\,\bigr)+ O(\hbar^\infty).
\]

\subsection{Variance estimates}\label{ss: 24} For $\hbar\in[0,\hbar_0]$ we define
\begin{equation}\label{psi}
\psi^\hbar:\bR^m\to\bR,\;\;\psi^\hbar(\bx)=\begin{cases}
\max_{|\alpha|\leq 4} \bigl|\,\pa^\alpha_{\bx} V^\hbar(\bx)\,\bigr|, & |\bx|_\infty\leq \frac{1}{2\hbar},\\
&\\
0, & |\bx|_\infty>\frac{1}{2\hbar}.
\end{cases}
\end{equation}
\begin{lemma}\label{lemma: psih} For any $p\in[0,\infty]$  we have
\begin{equation}\label{psih}
\bigl\Vert\,\psi^\hbar-\psi^0\,\bigr\Vert_{L^p(\bR^m)}= O(\hbar^\infty).
\end{equation}
\end{lemma}

\begin{proof}  We distinguish two cases.

\smallskip

\noindent {\bf 1.} $p=\infty$.     Note that  that (\ref{eq: vve_appr}) implies
\[
\sup_{|\bx|_\infty\leq 1/(2\hbar)}|\psi^\hbar(\bx)-\psi^0(\bx)|= O(\hbar^\infty).
\]
Since $V$ is a  Schwartz function we deduce that
\[
\sup_{|\bx|_\infty> 1/(2\hbar)}|\psi^\hbar(\bx)-\psi^0(\bx)|=\sup_{|\bx|_\infty> 1/(2\hbar)}|\psi^0(\bx)|=O(\hbar^\infty).
\]
\noindent {\bf 2.} $p\in[1,\infty)$.   We have
\[
\int_{\bR^m} |\psi^\hbar(\bx)-\psi^0(\bx)|^p d\bx=  \int_{|\bx|_\infty\leq 1/(2\hbar)} |\psi^\hbar(\bx)-\psi^0(\bx)|^p d\bx+\int_{|\bx|_\infty> 1/(2\hbar)}|\psi^0(\bx)|^pd\bx.
\]
The integrand in the first integral in  the right-hand side is $O(\hbar^\infty)$ and the volume  of the region is integration is $O(\hbar^{-m})$ so the first integral is $O(\hbar^\infty)$.  Since $V$ is a Schwartz function we deduce   that 
\[
|\psi^0(\bx)|= O\bigl(\, \vert \bx\vert^{-N}\,\bigr),\;\;\forall N\in\bN.
\]
This shows that the second integral is also $O(\hbar^\infty)$.
\end{proof}

\begin{proposition}\label{prop: el21}  There exists ${S}^0\in (0,\infty)$ such that
\[
\lim_{h\searrow 0}\var\bigl[\zeta^\hbar(2N_\hbar)\,\bigr] = {S}^0=\lim_{h\searrow 0}\var\bigl[\zeta^0(2N_\hbar)\,\bigr].
\]
\end{proposition}

\begin{proof} The  \cite{Niclt} we proved that  the limit
\begin{equation}\label{s0}
\lim_{N\to \infty}\var\bigl[\zeta^0(2N)\,\bigr]
\end{equation}
exists, it is finite and nonzero. We  denote by $S^0$ this  limit.  It remains to prove two facts.

\begin{itemize}

\item[($\mathbf{F}_1$)] The limit $\bar{S}^0:=\lim_{h\searrow 0}\var\bigl[\zeta^\hbar(2N_\hbar)\,\bigr]$  exists and it is finite.

\item[($\mathbf{F}_2$)]  $S^0=\bar{S}^0$.

\end{itemize}
To prove these facts, we  will employ  a  refinement of the strategy  used in the proof of  \cite[Prop.3.3]{Niclt}.

\medskip

\noindent{\bf Proof of $\mathbf{F}_1$}. Using (\ref{za}) we deduce 
\[
\zeta^\hbar(2N_\hbar)=(2N_\hbar)^{-m/2}\Bigl( \bsZ^\hbar(\whB_{2N_\hbar})-\bE\bigl[\,\bsZ^\hbar(\whB_{2N_\hbar})\,\bigr]\,\Bigr) =(2N_\hbar)^{-m/2}\sum_{q>0} Z_q^\hbar(\whB_{2N_\hbar}).
\]
We set
\[
S^\hbar:= \var\bigl[\, \zeta^\hbar(2N_\hbar)\,\bigr]=\bE\bigl[\, \zeta^\hbar(2N_\hbar)^2\,\bigr]=\sum_{q>0}\underbrace{(2N_\hbar)^{-m} \bE\bigl[\,Z^\hbar(\whB_{2N_\hbar})^2\,\bigr]}_{=: S^\hbar_q}
\]
To estimate $S^\hbar_q$   we write
\[
Z^\hbar_q(\whB_{2N_\hbar})=\int_{\whB_{2N_\hbar}} \rho^\hbar_q(\bx) d\bx,
\]
where $\rho^\hbar_q(\bx)$ is described in (\ref{zb}). Then
\[
S_q^\hbar=(2N_\hbar)^{-m}\int_{\whB_{2N_\hbar}\times \whB_{2N_\hbar} } \bE\bigl[ \,\rho^\hbar_q(\bx)\rho^\hbar_q(\by)\,\bigr] d\bx d\by
\]
(use the stationarity of $Y^\hbar(\bx)$)
\[
=(2N_\hbar)^{-m}\int_{\whB_{2N_\hbar}\times \whB_{2N_\hbar} } \bE\bigl[ \,\rho^\hbar_q(0)\rho^\hbar_q(\by-\bx)\,\bigr] d\bx d\by
\]
\[
=\int_{\whB_{4N_\hbar}} \bsE\bigl[ \,\rho^\hbar_q(0)\rho^\hbar_q(\bu)\,\bigr]\prod_{k=1}^m\left(1-\frac{|u_k|}{2N_\hbar}\right) d\bu.
\]
The last equality is  obtained by integrating along the fibers of the  map 
\[
\whB_{2N_\hbar}\times \whB_{2N_\hbar}\ni (\bx,\by)\mapsto \by-\bx\in \whB_{4N_\hbar}.
\]
At this point we need to invoke (\ref{zb}) to the effect that
\[
\rho_q^\hbar(\bx) =\sum_{\substack{(\alpha,\beta)\in\eI_m,\\ |\alpha|+|\beta|=q}} \omega_\hbar d_\alpha f^\hbar_\beta H_\alpha(U(\bx)) H_\beta(A(\bx)).
\]
We can rewrite   this in a more compact form.  Set
\[
\Xi^\hbar(\bx):=\bigl( U(\bx),\;A(\bx) \,\bigr).
\]
For $\gamma=(\alpha,\beta)\in \eI_m$ we set
\[
\ba^\hbar(\gamma):=  \omega_\hbar d_\alpha f^\hbar_\beta,\;\; H_\gamma(\Xi^\hbar(\bx)):=H_\alpha(U(\bx)) H_\beta(A(\bx)).
\]
Then
\begin{equation}\label{rhoq0}
\rho^\hbar_q(\bx)=\sum_{\gamma\in\eI_m,\;|\gamma|=q} \ba^\hbar(\gamma) H_\gamma\bigl(\,\Xi^\hbar(\bx)\,\bigr),
\end{equation}
\[
\bE\bigl[ \,\rho^\hbar_q(0)\rho^\hbar_q(u)\,\bigr]=\sum_{\substack{\gamma,\gamma'\in\eI_m\\ |\gamma|=|\gamma'|=q}} \ba^\hbar(\gamma)\ba^\hbar(\gamma')\bE\bigl[\,H_\gamma(\,\Xi^\hbar(0)\,) H_{\gamma'}(\,\Xi^\hbar(\bu)\,)\,\bigr].
\]
We set  $\omega(m):=m+\nu(m)$, and we denote by $\Xi_i(\bx)$, $1\leq i\leq  \omega(m)$, the components of $\Xi(\bx)$ labelled so  that $\Xi_i(\bx)=U_i(\bx)$, $\forall 1\leq i\leq m$.  For $\bu\in\bR^m$, $\hbar\in[0,\hbar_0]$  and $1\leq i,j\leq \omega(m)$ we define the covariances
\[
\Gamma^\hbar_{ij}(\bu):=\bE\bigl[\,\Xi^\hbar_i(0)\Xi^\hbar_j(\bu)\,\bigr].
\]
Using the Diagram Formula (see e.g.\cite[Cor. 5.5]{Maj} or \cite[Thm. 7.33]{Jan}) we deduce that  for any $\gamma,\gamma'\in\eI_m$ such that $|\gamma|=|\gamma'|=q$, there exists a \emph{universal} homogeneous polynomial of degree $q$, $P_{\gamma,\gamma'}$ in the variables $\Gamma_{ij}(\bu)$ such that
\[
\bE\bigl[\,H_\gamma(\,\Xi^\hbar(0)\,) H_{\gamma'}(\,\Xi^\hbar(\bu)\,)\,\bigr]=P_{\gamma,\gamma'}\bigl(\,\Gamma^\hbar_{ij}(\bu)\,\bigr).
\]
Hence
\begin{equation}\label{vqn1}
S^\hbar_q=(2N_\hbar)^{-m}\sum_{\substack{\gamma,\gamma'\in\eI_m\\
|\gamma|=|\gamma'|=q}} \ba^\hbar(\gamma)\ba^\hbar(\gamma')\;\underbrace{\int_{\whB_{4N_\hbar}}P_{\gamma,\gamma'}\bigl(\,\Gamma^\hbar_{ij}(\bu)\,\bigr) \prod_{k=1}^m\left(1-\frac{|u_k|}{2N_\hbar}\right) d\bu}_{=:R^\hbar(\bgamma,\bgamma')}.
\end{equation}
From  (\ref{ddag}) we deduce that 
\begin{equation}\label{whb}
\whB_{4N_\hbar}\subset \whB_{1/\hbar}
\end{equation} 
so that $\whB_{4N_\hbar}\subset \supp \psi^\hbar$, where $\psi^\hbar$ is the function defined in (\ref{psi}).   We deduce  that there exists a positive constant $K$, independent of $\hbar\in [0,\hbar_0]$, such that
\begin{equation}\label{psik}
\bigl|\,\Gamma^\hbar_{i,j}(\bu)\,\bigr|\leq K\psi^\hbar(\bu),\;\;\forall i,j=1,\dotsc, \omega(m),\;\;\bu\in\whB_{4N_\hbar}.
\end{equation}
From (\ref{psik}) we deduce that   for any $\gamma,\gamma'\in\eI_m$ such that $|\gamma|=|\gamma'|=q$ there exists a constant $C_{\gamma,\gamma'}>0$ such that
\[
\bigl|\, P_{\gamma,\gamma'}\bigl(\,\Gamma^\hbar_{ij}(\bu)\,\bigr) \,\bigr|\leq C_{\gamma,\gamma'}\psi^\hbar(\bu)^q,\;\;\forall \bu\in\whB_{4N_\hbar}.
\]
We know from (\ref{ast}) that $N_\hbar\to\infty$ as $\hbar\to 0$. Arguing exactly as in the proof of Lemma \ref{lemma: psih} we deduce   that
\begin{equation}\label{rnlim}
\lim_{\hbar\to 0} R^\hbar(\bgamma,\bgamma')=R^0(\gamma,\gamma'):= \int_{\bR^m}P_{\gamma,\gamma'}\bigl(\,\Gamma^0_{ij}(\bu)\,\bigr) d\bu, 
\end{equation}
and thus
\begin{equation}\label{vq2}
\lim_{\hbar\to 0} S_q^\hbar=\bar{S}^0_q:= \sum_{\substack{\gamma,\gamma'\in\eI_m\\
|\gamma|=|\gamma'|=q}} \ba^0(\gamma)\ba^0(\gamma') R^0(\gamma,\gamma')=\int_{\bR^m} \bE\bigl[ \,\rho^0_q(0)\rho^0_q(\bu)\,\bigr]d\bu.
\end{equation}
Since $S^\hbar_q\geq 0$, $\forall q, \hbar$, we have
\[
S^\hbar_q\geq 0,\;\;\forall q.
\]
We denote by $\eP_{>Q}$ the projection
\[
\eP_{>Q}=\sum_{q>Q}\eP_q,
\]
where  $\eP_q$ denotes the projection on the $q$-th chaos component of $\widehat{\eX}$.

\begin{lemma}\label{lemma: VQ} For any positive integer $Q$ we set
\[
S^\hbar_{>Q}:=\bE\Bigl[\,\bigr|\eP_{>Q}\zeta^\hbar(2N_\hbar)\bigr|^2\,\Bigr]=\sum_{q>Q}S^\hbar.
\]
Then
\begin{equation}\label{supN}
\lim_{Q\to\infty} \bigl(\,\sup_\hbar S_{>Q}\,\bigr) =0,
\end{equation}
the series 
\[
\sum_{q\geq 1}\bar{S}^0_q
\]
is convergent and, if $\bar{S}^0$ is its sum, then
\begin{equation}\label{vinfty}
\bar{S}^0=\lim_{\hbar\to 0} S^\hbar=\lim_{\hbar\to 0}\sum_{q\geq 1}  {S}^\hbar_{q}.
\end{equation}
\end{lemma}

\begin{proof} For $\bx\in\bR^m$ we denote by $\theta_\bx$ the shift operator associated with the stationary  fields  $\bsY^\hbar$, i.e.,
\[
\theta_\bx Y^\hbar(\bullet)= Y^\hbar(\bullet+\bx).
\] 
This extends to a unitary map $L^2(\Omega)\to L^2(\Omega)$  that commutes with the chaos decomposition of $L^2(\Omega)$. Moreover,  for any  box  $B$  and any $\hbar\in [0,\hbar_0]$ we have
\[
\bsZ^\hbar(B+\bx)=\theta_\bx \bsZ^\hbar(B).
\]
If we denote by $\eL_\hbar$ the set 
\begin{equation}\label{elh}
\eL_\hbar:=\bA^m\cap \whB_{4N_\hbar},
\end{equation}
then we deduce
\begin{equation}\label{b1}
\zeta^\hbar(2N_\hbar)=(2N_\hbar)^{-m/2}\sum_{\bp\in\eL_\hbar} \theta_\bp\zeta^\hbar\bigl(B),\;\;B=\whB_1.
\end{equation}
 We have
\[
\eP_{>Q}\zeta^\hbar(2N_\hbar)=(2N_\hbar)^{-m/2}\sum_{\bp\in\eL_\hbar} \theta_\bp \eP_{>Q}\zeta^\hbar\bigl(B).
\]
Using the stationarity of $Y^\hbar$ we deduce
\begin{equation}\label{V>Q}
S^\hbar_{>Q,}=\bE\Bigl[\,\bigr|\eP_{>Q}\zeta^\hbar(2N_\hbar)\bigr|^2\,\Bigr]=(2N_\hbar)^{-m}\sum_{\bp\in\eL_{\hbar}}\nu(\bp, N_\hbar)\bsE\bigl[\,\eP_{>Q}\zeta\bigl(B)\cdot\theta_\bp \eP_{>Q}\zeta^\hbar\bigl(B)\,\bigr],
\end{equation}
where $\nu(\bp,N_\hbar)$ denotes the number of  points $\bx\in\eL_\hbar$ such that $\bx-\bp\in\whB_{2N_\hbar}$. Clearly
\begin{equation}\label{nubs}
\nu(\bp, N_\hbar)\leq (2N_\hbar)^m.
\end{equation}
With $K$ denoting the positive constant in (\ref{psik})  we  deduce from Lemma \ref{lemma: psih} that we can choose  positive numbers $a,\rho$ such that
\[
\psi^\hbar(\bx)\leq \rho<\frac{1}{K},\;\;\forall |\bx|_{\infty}>a,\;\;\forall \hbar\in[0,\hbar_0].
\]
We split $S^\hbar_{>Q}$ into two parts,
\[
S^\hbar_{>Q}=S^\hbar_{>Q,0}+S^\hbar_{>Q,\infty},
\]
where  $S^\hbar_{>Q,0}$ is made up of  the terms in (\ref{V>Q})  corresponding to points  $\bp\in\eL_{\hbar}$ such that $|\bp|_\infty <a+1$,  while $S^\hbar_{>Q,\infty}$  corresponds to  points $\bp\in\eL_{\hbar}$ such that  $|\bp|_\infty\geq a+1$.

We  deduce from (\ref{nubs}) that for $2M>a+1$ we have
\[
\Bigr| \,S^\hbar_{>Q,0}\,\Bigr|\leq (2N_\hbar)^{-m}(2a+2)^m  (2N_\hbar)^m\bE\Bigl[\, \bigl|\eP_{>Q}\zeta^\hbar(B)\bigr|^2\,\Bigr].
\]
Proposition \ref{prop: key}(iv) implies that, as $Q\to \infty$, the right-hand side of the above inequality goes to $0$ uniformly with respect to $\hbar$.

To estimate $S^\hbar_{>Q,\infty}$ observe that for $\bp\in \eL_{\hbar}$ such that $|\bp|_\infty>a+1$ we have
\begin{equation}\label{P>Q}
\bsE\bigl[\,\eP_{>Q}\zeta^\hbar(B)\cdot\theta_\bp \eP_{>Q}\zeta^\hbar(B)\,\bigr]=\sum_{q>Q}\int_B\int_B\bE\bigl[\,\rho^\hbar_q(\bx)]\rho^\hbar_q(\by+\bp)\,\bigr] d\bx d\by,
\end{equation}
where we recall from (\ref{rhoq0}) that
\[
\rho^\hbar_q(\bx)=\sum_{\gamma\in\eI_m,\;|\gamma|=q} \ba^\hbar(\gamma) H_\gamma\bigl(\,\Xi^\hbar(\bx)\,\bigr),\;\;\eI_m:=\bN^m_0\times\bN_0^{\nu(m)},\;\;\nu(m)=\frac{m(m+1)}{2}.
\]
Thus
\[
\bE\bigl[\,\rho_q(\bx)\rho_q(\by+\bp)\,\bigr]=\bE\Bigl[\,\Bigl(\sum_{\substack{\gamma\in\eI_m,\\|\gamma|=q}} \ba^\hbar(\gamma) H_\gamma\bigl(\,\Xi^\hbar(\bx)\,\bigr)\,\Bigr)\Bigl(\, \sum_{\substack{\gamma\in\eI_m,\\ |\gamma|=q}} \ba^\hbar(\gamma) H_\gamma\bigl(\,\Xi^\hbar(\by+\bp)\,\bigr)\,\Bigr)\,\Bigr]
\]
Arcones' inequality \cite[Lemma 1]{Arc} implies that
\begin{equation}\label{P>Q1}
\bE\bigl[\,\rho_q(\bx)\rho_q(\by+\bp)\,\bigr]\leq  K^q\psi^\hbar(\bp+\by-\bx)^q \sum_{\substack{\gamma\in\eI_m,\\ |\gamma|=q}} |\ba^\hbar(\gamma)|^2\gamma!.
\end{equation}
We are not out of the woods yet since the series $\sum_{\gamma\in\eI_m} |\ba^\hbar(\gamma)|^2\gamma !$ is divergent.  To bypass this issue observe taht, for $\gamma=(\alpha,\beta)\in\eI_m$ we  have
\[
\ba^\hbar(\gamma)= \omega_\hbar d_\alpha f^\hbar_\beta
\]
where, according to (\ref{dalph}) we have $d_\alpha=\frac{1}{\alpha!}(2\pi )^{-\frac{m}{2}}  H_\alpha(0)$. Recalling that
\[
H_{2r}(0)= (-1)^r\frac{(2r)!}{2^rr!},\;\;H_{2r+1}(0)=0.
\]
we deduce that
\[
 (2r)!\Bigl|\frac{1}{(2r)!}H_{2r}(0)\Bigr|^2=\frac{1}{2^{2r}}\binom{2r}{r}\leq 1, 
\]
and
\[
d_\alpha^2\alpha!\leq C=\frac{1}{(2\pi)^{m/2}}.
\]
Using (\ref{lam}) and (\ref{fn}) we conclude that 
\[
 \sum_{\substack{\gamma\in\eI_m,\\ |\gamma|=q}} |\ba^\hbar(\gamma)|^2\gamma!\leq \omega_\hbar (2\pi)^{-m/2} q^m \sum_{\substack{\beta\in\bN_0^{\nu(m)},\\ |\beta|\leq q}} (f^\hbar_\beta)^2\beta!\leq C q^m\bE\bigl[\,|\det \nabla^2 Y^\hbar(0)|^2\,\bigr].
\]
Using this in (\ref{P>Q}) and (\ref{P>Q1}) we deduce
\[
\bE\bigl[\,\eP_{>Q}\zeta^\hbar(B)\cdot\theta_\bp \eP_{>Q}\zeta^\hbar(B)\,\bigr]
\]
\[
\leq \underbrace{C\bE\bigl[\,|\det\nabla Y^\hbar(0)|^2\,\bigr]}_{=:C'}\; \sum_{q>Q} q^mK^q\int_B \int_B\psi^\hbar(\bs+\bu-\bt)^q d\bu d\bt
\]
Hence
\[
\bigl|\,S^\hbar_{>Q, \infty}\bigr|\leq C'\Bigl( \sum_{q>Q}  q^mK^q\rho^{q-1}\Bigr)\Bigl( \sum_{\substack{\bp\in\eL_{\hbar},\\ |\bs|_\infty>a+1}}\int_B\int_B \psi^\hbar(\bp+\by-\bx) d\by d\bx\Bigr),
\]
where we have used  the fact that for $|p|_\infty\geq a+1$, $|\by|,|\bx|\leq 1$  we have  $\psi^\hbar(\bp+\by-\bx)<\rho$. Since $\rho<\frac{1}{K}$, the sum 
\[
\sum_{q>Q}  q^mK^q\rho^{q-1}
\]
is the tail of a convergent power series.   On the other hand,
\[
\sum_{\substack{\bp\in\eL_{\hbar},\\ |\bp|_\infty>a+1}} \int_B\int_B \psi^\hbar(\bp+\by-\bx) d\by d\bx\leq\sum_{\bp\in\eL_{\hbar}}\int_{[-1,1]^m}\psi^\hbar(\bp+\by)
\]
\[
\leq 2\int_{\bR^m}\psi^\hbar(\by) d\by\stackrel{(\ref{psih})}{=}O(1).
\]
This proves   that $\sup_\hbar|S^\hbar_{>Q, \infty}|$ goes to zero as $Q\to\infty$ and  completes the proof of  (\ref{supN}). The claim (\ref{vinfty}) follows immediately from (\ref{supN}). This concludes the proof of  Lemma \ref{lemma: VQ} and of the  fact $\mathbf{F}_1$. \end{proof}

\medskip

\noindent{\bf Proof of $\mathbf{F}_2$}. In \cite{Niclt} we  have shown that  the limit  $S^0$ in (\ref{s0}) is the sum of a series
\[
S^0=\sum_{q\geq 1} S^0_q,\;\; S^0_q= \int_{\bR^m} \bE\bigl[ \,\rho^0_q(0)\rho^0_q(\bu)\,\bigr]d\bu.
\]
The equality (\ref{vq2})  shows that $\bar{S}^0=S^0>0$. This concludes the proof of Proposition \ref{prop: el21}. \end{proof}

\subsection{Proof of Theorem \ref{th: main2}}\label{ss: 25} In \cite{Niclt} we have shown that, as $\hbar\to 0$, the random variables converge in law to a random variable $\sim\eN(0, S^0)$. 

As explained in Subsection \ref{ss: 14}, to conclude the proof of Theorem \ref{th: main2} it suffices to establish  the asymptotic normality as $\hbar \to 0$ of the family
 \[
\zeta^\hbar_{q}=\frac{1}{(2N_\hbar)^{m/2}} \int_{\whB_{2N_\hbar}} \rho_q(\bx) d\bx,\;\;\forall q\geq 1.
 \]
This follows from the fourth-moment theorem \cite[Thm. 5.2.7]{NP}, \cite{NuaPe}.      Here are the details. 

Recall  from \cite[IV.1]{Jan} that we have a surjective  isometry $\Theta_q: \eX^{\odot q}\to\eX^{:q:}$ , where  $\eX^{\odot q}$ is the $q$-th symmetric power  and $\eX^{:q:}$  is the $q$-th  chaos component of $\widehat{\eX}$.  The multiple  Ito integral $\bsI_q$  is then the map
\[
\bsI_q=\frac{1}{\sqrt{q!}} \Theta_q.
\]
We can write $\zeta^\hbar_{q}$  as a multiple  Ito integral
\[
\zeta^\hbar_{q} = \bsI_q\bigl[\,g_q^\hbar\,\bigr], \;\;g_q^\hbar\in\eX^{\odot n}.
\]
According to  \cite[Thm.5.2.7(v)]{NP}, to prove that $\zeta^\hbar_{q}$ converge in law to  a normal variable it suffices to show that
\begin{equation}\label{EL18}
\lim_{\hbar\to 0}\Vert g_q^\hbar\otimes_r g_q^\hbar\Vert_{\eX^{\otimes(2q-2r)}}=0,\;\;\forall r=1,\dotsc, q-1.
\end{equation}
In our context, using the isometry $\bsI$ in (\ref{ito})  we can view $g_q^\hbar$ as a function 
\[
g_q^\hbar\in L^2\bigl( \, (\bR^m\times \bR^m)^q\,\bigr) ,\;\;g_q^\hbar=g_q^\hbar(\bz_1,\dotsc,\bz_q),\;\;\bz_j\in\bR^m\times \bR^m,
\]
and then
\[
g_q^\hbar\otimes_r g_q^\hbar\in  L^2\bigl( \, (\bR^m\times \bR^m)^{2(q-r)}\,\bigr),\;\;g_q^\hbar\otimes_r g_q^\hbar(\bz_{q-r+1},\bz'_{q-r+1},\dotsc ,\bz_q,\bz'_q)
\]
\[
=\int_{(\bR^m\times\bR^m)^r}  g_q^\hbar(\bz_1,\dotsc,\bz_q, \bz_{q-r+1},\dotsc,\bz_q) g_q^\hbar(\bz_1,\dotsc,\bz_q, \bz'_{q-r+1},\dotsc,\bz'_q) d\bz_1\cdots d\bz_q.
\]
 To show  (\ref{EL18}) we invoke  the  arguments  following  the inequality (18) in the second step of the proof of \cite[Prop. 2.4]{EL}  which extend   to the setup in this paper.

\subsection{Proof of Theorem \ref{th: main3}} \label{ss: cor}    We set 
\[
N_\hbar:=\left\lfloor \,\frac{r}{2\hbar}\,\right\rfloor,\;\;s_\hbar:=\frac{ r}{2\hbar N_\hbar}.
\]
Then  $N_\hbar\in \bN$,
\[
\lim_{\hbar \to 0} N_\hbar=\infty,\;\; \whB_{2N_\hbar s_\hbar} =\whB_{\hbar^{-1}r}\subset \whB_{1/(2\hbar)},\;\;\lim_{\hbar\to 0} s_\hbar=1.
\]
Thus, $\whB_{\hbar^{-1}r}$  is  a cube, centered at $0$  with vertices in the lattice $(s_\hbar\bZ)^m$ and $s(\hbar)\approx 1$ for $\hbar$ small.

To reach  the conclusion (i)   run the arguments   in the the proof of Theorem \ref{th: main1} with the following modified notations:  the box $\whB_{2N_\hbar}$ should be replaced with the box $\whB_{2N_\hbar s_\hbar}=s_\hbar\whB_{2N_\hbar}$,  the lattice $\bA^m$ in (\ref{ash}) replaced by $s_\hbar\bA^m$, and $\whB(\ba)$ redefined as $s_\hbar\whB_1(\ba)=\whB_{s_\hbar}(\ba)$.

To reach  the conclusions (ii) and (iii)  of  Theorem \ref{th: main3} run the arguments   in the  subsections \ref{ss: 24} and \ref{ss: 25} with the following modified notations:  the box $\whB_{2N_\hbar}$ should be replaced with the box $\whB_{2N_\hbar s_\hbar}=s_\hbar\whB_{2N_\hbar}$, in (\ref{elh}) the set $\eL_\hbar$ should be redefined to be
\[
\eL_\hbar:= s_\hbar\bigl( \bA^m\cap \whB_{4N_\hbar}\,\bigr),
\]
and  the box $B$ in (\ref{b1}) should be redefined to be $\whB_{s_\hbar}=s_\hbar\whB_1$.\qed 

\begin{remark} The above proof shows that for any $r\in (0,1/2]$ we have
\begin{equation}\label{supNa}
\lim_{Q\to\infty} \left(\,\sup_\hbar \bE\Bigl[\,\bigr|\eP_{>Q}\zeta^\hbar(r/\hbar)\bigr|^2\,\Bigr]\,\right) =0.
\end{equation}\qed 
\end{remark}

\subsection{Proof of Theorem \ref{th: main4}}\label{ss: main4}  Note that (see  (\ref{nota}) for notation)
\[
\bsZ(X^\hbar, \bT^m)=\bsZ^\hbar(1/h).
\]
Hence it suffices to prove
\begin{enumerate} 

\item As $\hbar\to 0$ 
\[
\bsZ^\hbar(1/h)= \hbar^{-m}\bigl(\,  \bar{Z}_0+O(\hbar^\infty)\,\bigr).
\]
\item As $\hbar\to 0$  
\[
\var\bigl[\, \bsZ^\hbar(1/h)\,\bigr]\sim S^0 \cdot (1/\hbar)^m.
\]

\item The families of random variables
\[
\Bigl\{\,\zeta^\hbar(1/\hbar)\,\Bigr\}_{\hbar\in(0,\hbar_0]}\;\;\mbox{and}\;\;\Bigl\{\, \zeta^0(1/\hbar)\,\Bigr\}_{\hbar\in(0,\hbar_0]}
\]
converge in distribution  as $\hbar\to 0$ to normal random variables $\sim \eN(0, S^0)$.
\end{enumerate}

From this point of view, these results  extend  Theorem \ref{th: main3}) to the case $r=1$. We cannot invoke  Theorem \ref{th: main3} (or Theorem \ref{th: main2}) with $2N_\hbar$ formally replaced by $\frac{1}{\hbar}$)  to prove (ii) and (iii) because the condition (\ref{ddag}) is violated.  However, a slight modification of the arguments in the proof of  Theorem \ref{th: main2} will  yield the desired conclusion.        

Let us first observe that it suffices to show that the variables $\zeta^\hbar(1/\hbar)$ satisfy  the  conclusions of Proposition \ref{prop: el21}.  Once  we verify this fact, the arguments  in  Subsection \ref{ss: 25}  extend without any modification to this case yielding  Theorem \ref{th: main4}.

Let us first show ${\bf F}_1$, i.e., the limit 
\[ 
\bar{S}^0:=\lim_{h\searrow 0}\var\bigl[\zeta^\hbar(1/h)\,\bigr]
\]
exists  and it is finite.   Using (\ref{vqn1}) (with $2N_\hbar$ replaced by $1/\hbar$) we deduce
\begin{equation}\label{vqn1a}
S^\hbar_q=\eP_q\zeta(1/\hbar)=\sum_{\substack{\gamma,\gamma'\in\eI_m\\
|\gamma|=|\gamma'|=q}} \ba^\hbar(\gamma)\ba^\hbar(\gamma')\;\int_{\whB_{2/\hbar}}P_{\gamma,\gamma'}\bigl(\,\Gamma^\hbar_{ij}(\bu)\,\bigr) \prod_{k=1}^m(1-\hbar|u_k|)  d\bu.
\end{equation}
We set
\[
\bOm_m:=\{\pm 1\}^m\subset \bR^m.
\]
For $\bom\in\bOm_m$ we   denote by $\eO^\bom$ the ``octant'' in $\bR^m$ that contains $\bom$, i.e., 
\[
\eO^\bom:=\bigl\{ \bx\in\bR^m;\;\;x_i\omega_i>0,\;\;\forall i\,\bigr\},\;\;\whB^\bom_R:=\whB_R\cap \eO^\bom.
\] 
Note that
\[
\int_{\whB_{2/\hbar}}P_{\gamma,\gamma'}\bigl(\,\Gamma^\hbar_{ij}(\bu)\,\bigr) \prod_{k=1}^m\bigl(1-\hbar|u_k|\bigr)  d\bu=\sum_{\bom\in\bOm_m} \int_{\whB_{2/\hbar}^\bom}P_{\gamma,\gamma'}\bigl(\,\Gamma^\hbar_{ij}(\bu)\,\bigr) \prod_{k=1}^m\bigl(1-\hbar|u_k|\bigr)  d\bu.
\]
For any $\bom\in\bOm_m$ the cube $\whB^\bom_{2/\hbar}$  it is centered at  $\frac{1}{2\hbar}\bom$ and it  has size $\frac{1}{\hbar}$. For any $\bnu\in\bOm_m$ the cube $\whB_{1/\hbar}$ has a unique  vertex in $\eO^\bnu$, namely $\frac{1}{2\hbar}\bnu$.  We obtain a decomposition (see Figure \ref{fig: 1})
\begin{figure}[ht]
\centering{\includegraphics[height=4.3in,width=3.2in]{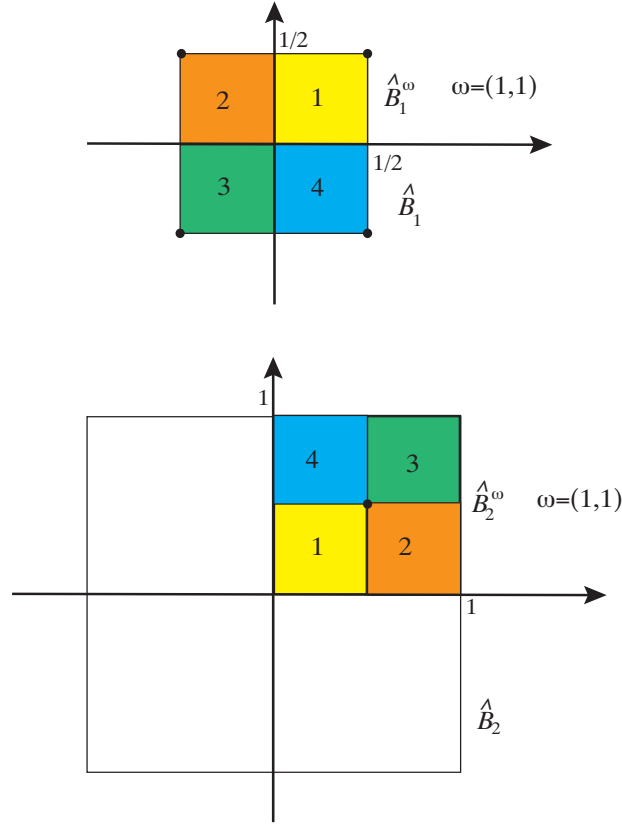}}
\caption{\sl The decomposition (\ref{dec}) in the case $m=2$, $\hbar=1$, $\bom=(1,1)$. The $2^m$ quadrants decompose $\whB_1$ into $2^m$ half-sized squares and that can be reassembled back into $\whB_2^\bom$  via the $2^m$ translations that  map the a vertex  of $\whB_1$  to the center of $\whB_2^\bom$.}
\label{fig: 1}
\end{figure}

\begin{equation}\label{dec}
\whB^\bom_{2/\hbar}\setminus H =\bigcup_{\bnu\in\bOm_m}\Bigl( \frac{1}{2\hbar}(\bom-\bnu)+\whB^\bnu_{1/\hbar}\,\Bigr),\;\;H=\bigcup_{i=1}^m\{x_i=0\}.
\end{equation}
 Denote by $L_\hbar$ the lattice $(\hbar^{-1}\bZ)^m$.  Note that the integrand  $P_{\gamma,\gamma'}\bigl(\,\Gamma^\hbar_{ij}(\bu)\,\bigr) $ is $L_\hbar$-periodic because $V^\hbar$ is such.  Since 
\[
\frac{1}{2\hbar}(\bom-\bnu)\in L_\hbar,\;\;\forall \bom,\bnu\in\bOm_m,
\]
we deduce
\[
 \int_{\whB_{2/\hbar}^\bom}P_{\gamma,\gamma'}\bigl(\,\Gamma^\hbar_{ij}(\bu)\,\bigr) \prod_{k=1}^m(1-\hbar|u_k|)  d\bu
 \]
 \[
 =\sum_{\bnu\in\bOm_m} \int_{\whB_{1/\hbar}^\bnu}P_{\gamma,\gamma'}\bigl(\,\Gamma^\hbar_{ij}(\bv)\,\bigr) \prod_{k=1}^m\left(1-\hbar\left|v_k+\frac{1}{2\hbar}(\omega_k-\nu_k)\right|\right)  d\bv
 \]
\[
=\int_{\whB_{1/\hbar}}P_{\gamma,\gamma'}\bigl(\,\Gamma^\hbar_{ij}(\bv)\,\bigr)  f^\hbar_\bom(\bv) d\bv,
\]
where
\[
f^\hbar_\bom(\bv)\Bigl|_{\eO^\nu}:= \prod_{k=1}^m\left(1-\hbar\left|v_k+\frac{1}{2\hbar}(\omega_k-\nu_k)\right|\right),
\]
and $f^\hbar_\bom(\bv)=0$ if $\bv$ lies on one of the coordinate hyperplanes $\{x_i=0\}\subset \bR^m$. Hence
\begin{equation}\label{2hbar}
\int_{\whB_{2/\hbar}}P_{\gamma,\gamma'}\bigl(\,\Gamma^\hbar_{ij}(\bu)\,\bigr) \prod_{k=1}^m(1-\hbar|u_k|)  d\bu=\int_{\whB_{1/\hbar}}P_{\gamma,\gamma'}\bigl(\,\Gamma^\hbar_{ij}(\bv)\,\bigr) \left(\sum_{\bom\in\bOm_m} f^\hbar_\bom(\bv)\right) d\bv.
\end{equation}
Now observe that
\begin{subequations}
\begin{equation}\label{fboma}
|f^\hbar_\bom(\bv)|\leq  3^m,\;\;\forall \bom,\;\;\forall \bv\in\bR^m,
\end{equation}
\begin{equation}\label{fbomb}
\lim_{\hbar\searrow 0} f^\hbar_\bom(\bv)=f^\bom_0(\bv)=\bone_{\eO^\bom}.
\end{equation}
\end{subequations}
Using (\ref{2hbar}),(\ref{fboma}), (\ref{fbomb}) and  (\ref{psik})  we deduce exactly as in the proof of Lemma \ref{lemma: psih}   that
\begin{equation}\label{rnlim1}
\lim_{\hbar\searrow 0}\int_{\whB_{2/\hbar}}P_{\gamma,\gamma'}\bigl(\,\Gamma^\hbar_{ij}(\bu)\,\bigr) \prod_{k=1}^m(1-\hbar|u_k|)  d\bu=R^0(\gamma,\gamma'):= \int_{\bR^m}P_{\gamma,\gamma'}\bigl(\,\Gamma^0_{ij}(\bu)\,\bigr) d\bu.
\end{equation}
Hence
\begin{equation}\label{vq2a}
\lim_{\hbar\to 0} S_q^\hbar=\bar{S}^0_q:= \sum_{\substack{\gamma,\gamma'\in\eI_m\\
|\gamma|=|\gamma'|=q}} \ba^0(\gamma)\ba^0(\gamma') R^0(\gamma,\gamma')=\int_{\bR^m} \bE\bigl[ \,\rho^0_q(0)\rho^0_q(\bu)\,\bigr]d\bu.
\end{equation}
This is the counterpart of (\ref{vq2}).

Let us explain  how to prove Lemma \ref{lemma: VQ} in this context when (\ref{ddag}) is not satisfied.   We set
\[
\bOm_m^\hbar=\frac{1}{4\hbar}\cdot\bOm_m=\frac{1}{2}N_\hbar\bOm_m
\]
The point $(4\hbar)^{-1}\bom$  in $\bOm_M^\hbar$ the  centers of the cube  $\whB^\bom_{1/\hbar}=\whB_{1/\hbar}\cap\eO^\bom$. We have
\[
\zeta(1/\hbar)=(\hbar)^{m/2}\Bigl(\,\bsZ^\hbar(\whB_{1/h})-\bE\bigl[\,\bsZ^\hbar(\whB_{1/\hbar})\,\bigr]\,\Bigl)
\]
From the equality
\[
\whB_{1/\hbar}=\bigcup_{\bom^\hbar\in\bOm_m^\hbar} \whB_{1/2\hbar}(\bom^\hbar)
\]
we deduce
\[
\bsZ^\hbar(\whB_{1/h})-\bE\bigl[\,\bsZ^\hbar(\whB_{1/\hbar})\,\bigr]=\sum_{\bom^\hbar\in\bOm_m^\hbar}\theta_{\bom^\hbar} \bsZ^\hbar(\whB_{1/(2\hbar)})-\bE\bigl[\,\bsZ^\hbar(\whB_{1/(2\hbar)})\,\bigr],
\]
Since
\[
\bigl|\,\whB_{1/(2\hbar)}\,\bigr|=(2\hbar)^{-m},
\]
we deduce
\[
\zeta(1/\hbar)=\frac{1}{2^{m/2}} \sum_{\bom^\hbar\in\bOm_m^\hbar}\theta_{\bom^\hbar}\zeta(\whB_{1/(2\hbar)})
\]
In particular
\[
\eP_{>Q}\zeta(1/\hbar)=\frac{1}{2^{m/2}} \sum_{\bom^\hbar\in\bOm_m^\hbar}\theta_{\bom^\hbar}\eP_{>Q}\zeta(\whB_{1/(2\hbar)}).
\]
Hence
\[
\bE\Bigl[\,\Bigl| \eP_{>Q}\zeta(1/\hbar)\,\Bigr|^2\,\Bigr]=\frac{1}{2^m}\sum_{\bom^\hbar,\bnu\hbar\in\bOm_m^\hbar}\bE\Bigl[\theta_{\bom^\hbar}\eP_{>Q}\zeta(\whB_{1/(2\hbar)})\cdot \theta_{\bom^\hbar}\eP_{>Q}\zeta(\whB_{1/(2\hbar)})\,\Bigr]
\]
\[
\leq \frac{1}{2^m} \sum_{\bom^\hbar,\bnu\hbar\in\bOm_m^\hbar}\bE\Bigl[\,\Bigr|\eP_{>Q}\zeta(\whB_{1/(2\hbar)})\Bigl|^2\,\Bigr]=2^m  \bE\Bigl[\,\Bigr|\eP_{>Q}\zeta(\whB_{1/(2\hbar)})\Bigl|^2\,\Bigr].
\]
Using (\ref{supNa})   we deduce
\[
\lim_{Q\to\infty} \left(\,\sup_\hbar \bE\Bigl[\,\bigr|\eP_{>Q}\zeta^\hbar(\whB_{1/(2\hbar)})\bigr|^2\,\Bigr]\,\right) =0
\]
showing that
\[
\lim_{Q\to\infty }\sup_{\hbar} \bE\Bigl[\,\Bigl| \eP_{>Q}\zeta(1/\hbar)\,\Bigr|^2\,\Bigr]=0.
\]
This shows that  the variables $\zeta^\hbar(1/\hbar)$ satisfy the conclusions of Lemma \ref{lemma: VQ}. 

This shows that the limit  $\lim_{h\searrow 0}\var\bigl[\zeta^\hbar(1/h)\,\bigr]$ exists and it is finite. Moreover (\ref{vq2a}) shows that this limit is $S_0$.
\qed

\appendix

\section{Proof  of Proposition \ref{prop: key}}
\label{s: key}
\setcounter{equation}{0}

We will follow the strategy  in the proof  of \cite[Prop. 1.1]{EL}. Some modifications are required since  the random functions $Y^\hbar$ are not isotropic for $\hbar>0$.

Denote by $p^\hbar_\bx(-)$ and $p^\hbar_{\bx,\by}(-,-)$ the probability densities of the  Gaussian vectors  $\nabla Y^\hbar(\bx)$ and respectively $(\,\nabla Y^\hbar(\bx),\nabla Y^\hbar(\by)\,)$. For simplicity we denote by $|S|$ the Lebesgue volume of a  Borel subset    $S\subset \bR^m$.

Due to the stationarity of $Y^\hbar$ it suffices to assume that the box $B$ is centered at $0$. The  Gaussian random function $Y^\hbar$ is stationary. Using the Kac-Rice formula \cite[Ch.11]{AT}  or \cite[Ch.6]{AzWs} we deduce that, $\forall \bv\in\bR^m$ we have
\begin{equation}\label{KR1}
\bE\bigl[\, \bsZ^\hbar(\bv, B)\,\bigr]= \bE\bigl[\, |\det \nabla^2Y^\hbar(0)\,|\,\bigr]p_0(\bv)\,|B|,
\end{equation}
\begin{multline}\label{KR2}
\bE\bigl[\, \bsZ^\hbar(\bv, B)\bigl(\,\bsZ^\hbar(\bv, B)-1\bigr) \,\bigr]\\
= \int_{2B}\bigl|B\cap (B-\by)\bigr| \;\underbrace{\bsE_{\by,\bv}\Bigl[\,|\det \nabla^2Y^\hbar(0) \det \nabla^2Y^\hbar(\by)|\;\Bigr]}_{g^\hbar(\bv,\by)}\;p_{0,\by}(\bv,\bv) d\by, 
\end{multline}
where, for typographical reasons, we denoted by   $\bE_{\by,\bv}$  the conditional expectation 
\[
\bE_{\by,\bv}[-]=\bE\bigl[\,-| C_\by(\bv)\,\bigr],\;\;C_\by(\bv):=\{\nabla Y^\hbar(0)=\nabla Y^\hbar(\by)=\bv\}.
\]
The two sides of the equality  (\ref{KR2})   are simultaneously finite or infinite.     Let us point out that the integrand on the right-hand side  of this equality could blow-up  at $\by=0$ because the  Gaussian vector  $(\,\nabla Y^\hbar(0),\nabla Y^\hbar(0)\,)$ is degenerate and therefore
\[
\lim_{\by\to 0}p_{0,\by}(\bv,\bv)=\infty.
\]
The  most  demanding    part  in the proof of Proposition \ref{prop: key}   is showing that the right-hand side of (\ref{KR2})   is finite. This boils down to understanding the singularity at the origin of   the integrand in (\ref{KR2}).      In \cite{Nifluct} we proved this fact in the case $\bv=0$.   To deal with the general  case we will use   a blend of the ideas in \cite{EL} and \cite{Nifluct}.

\smallskip

\noindent {\bf Step 1.}    We will show that there exist $\hbar_1>0$, $r_1>0$ and  $C_1>0$  such that for any $ \hbar\leq h_1$ we have
\begin{subequations}
\begin{equation}\label{prob_densa}
p^\hbar_{0,\by}(0,0)<\infty,\;\;\forall \by\neq 0,
\end{equation}
\begin{equation}\label{prob_densb}
0< p_{0,\by}(\bv,\bv)\leq C_1\vert \by\vert^{-m},\;\;\forall 0< \vert \by\vert<r_1,\;\;\forall\bv\in\bR^m.
\end{equation}
\end{subequations}
These two facts follow from \cite[Lemma 3.5]{Nifluct} and  the obvious inequality
\[
p_{0,\by}(\bv,\bv)\leq p_{0,\by}(0,0).
\]

\smallskip

\noindent{\bf Step 2.}  We will show that there exist $\hbar_2>0$, $r_2>0$ and  $C_2>0$  such that for any $ \hbar\leq h_1$ we have
\begin{equation}\label{integr}
|g^\hbar(\bv,\by)|\leq C_2\Vert \by\Vert^2,\;\;\forall \vert \by\vert \leq r_2,\;\;\bv\in\bR^m.
\end{equation}
We set
\[
f^\hbar(\by,\bv):= \bsE_{\by,\bv}\Bigl[\,|\det \nabla^2Y^\hbar(0)|^2\;\Bigr].
\]
From the Cauchy inequality and the stationarity of $Y^\hbar$  we deduce
\[
g^\hbar(\bv,\by)^2\leq   \bsE_{\by,\bv}\Bigl[\,|\det \nabla^2Y^\hbar(0)|^2\;\Bigr]\cdot  \bsE_{\by,\bv}\Bigl[\,|\det \nabla^2Y^\hbar(\by)|^2\;\Bigr]
\]
\[
=f^\hbar(\by,\bv)f^\hbar(-\by,\bv)=f^\hbar(\by,\bv)^2.
\]
We now invoke  Hadamard's inequality \cite[Thm. 7.8.1]{HJ}:   \textit{if $A:\bR^m\to\bR^m$ is an $m\times m$ symmetric  positive operator and $\{ \be_1,\dotsc, \be_m\}$ is an   orthonormal basis of $\bR^m$, then}
\[
\det A\leq  \prod_{j=1}^m (A\be_j,\be_j).
\]
Applying this inequality to $A=\nabla^2 Y^h(0))^2$ and a fixed  orthonormal basis  $\{\be_1,\dotsc,\be_m\}$ such that
\[
\be_1:=\vert \by\vert^{-1}\by,
\]
we deduce 
\[
|\det \nabla^2Y^\hbar(0)|^2\leq \vert \by\vert^{-2}\Vert\nabla^2 Y^y(0)\by\Vert^2  \Vert\nabla^2 Y^\hbar(0)\Vert^{2(m-1)}.
\]
Hence
\begin{multline}\label{det_had}
|g^\hbar(\bv,\by)|\leq f^\hbar(\bv,\by) \bE_{\by,\bv}\Bigl[\,|\det \nabla^2Y^\hbar(0)|^2\;\Bigr]^2\\
 \leq  \vert \by\vert^{-2}\bE_{\by,\bv}\Bigl[\,\Vert\nabla^2 Y^y(0)\by\Vert^4\;\Bigr]^{\frac{1}{2}} \bE_{\by,\bv}\Bigl[\, \Vert\nabla^2 Y^y(0)\Vert^{4(m-1)}\;\Bigr]^{\frac{1}{2}}
 \end{multline}
Now observe that
\[
\Vert\nabla^2 Y^y(0)\by\Vert^2= \vert\by\vert^2 \sum_{j=1}^m Y^\hbar_{1j}(0)^2,
\]
where, for any smooth function $F:\bR^m\to\bR$, we set
\[
F_{i,j,...k}:=\pa_{\be_i}\pa_{\be_j}\cdots\pa_{\be_k} F.
\]
Thus
\[
\vert\nabla^2 Y^y(0)\by\vert^4 \leq m \vert \by\vert^4 \sum_{j=1}^m Y^\hbar_{1j}(0)^4,
\]
\[
\bE_{\by,\bv}\Bigl[\,\vert\nabla^2 Y^y(0)\by\vert^4\;\Bigr]\leq  m \vert \by\vert^4 \sum_{j=1}^m\bE_{\by,\bv}\Bigl[\,  Y^\hbar_{1j}(0)^4\,\Bigr],
\]
\begin{equation}\label{det_had1}
g^\hbar(\bv,\by)\leq \sqrt{m} \left(\sum_{j=1}^m\bE_{\by,\bv}\Bigl[\,  Y^\hbar_{1j}(0)^4\,\Bigr]\right)^{\frac{1}{2}}\bE_{\by,\bv}\Bigl[\, \Vert\nabla^2 Y^\hbar(0)\Vert^{4(m-1)}\,\Bigr]^{\frac{1}{2}}.
\end{equation}
For each $j=1,\dotsc, m$ define the random function
\[
F_j:[0,1]\to\bR,\;\; F_j(t)= Y_j^\hbar(t\by).
\]
Then
\[
F_j'(t)= \vert \by\vert Y^\hbar_{1,j}(t\by) ,\;\;F_j''(t)=\vert \by\vert^2 Y^\hbar_{1,1,j}(t\by)
\]
Using the Taylor formula with integral remainder we deduce
\[
F_j(1)-F_j(0)= F'_j(0)+\int_0^1F_j''(t)(1-t) dt,
\]
i.e.,
\[
 Y^\hbar_{j}(\by)-Y^\hbar_{j}(0)= Y^\hbar_{1j}(0) +\vert \by\vert \int_0^1Y^\hbar_{11j}(t\by)(t\by) dt.
\]
Hence
\[
Y^\hbar_{1,j}(0)= Y^\hbar_{j}(\by)-Y^\hbar_{j}(0)-\vert \by\vert \int_0^1Y^\hbar_{1,1,j}(t\by)(t\by) dt.
\]
Setting $v_j:=(\bv,\be_j)$ and observing that under the condition $C_\by(\bv)$ we have 
\[
Y^\hbar_j(0)=Y^\hbar_j(\by)=v_j, \;\; \forall j=1,\dotsc, m,
\]
 we deduce
\[
\bE_{\by,\bv}\Bigl[\,  Y^\hbar_{1,j}(0)^4\,\Bigr]=\bE_{\by,\bv}\left[\left( Y^\hbar_{j}(\by)-Y^\hbar_{j}(0)-\vert \by\vert \int_0^1Y^\hbar_{1,1,j}(t\by)(1-t) dt\right)^4\,\Bigr|\; C_{\by}(\bv)\right]
\]
\[
=\vert \by\vert^4\bE_{\by,\bv}\left[\left( \int_0^1Y^\hbar_{1,1,j}(t\by)(1-t) dt\right)^4\,\right] \leq \vert \by\vert^4\bE_{\by,\bv}\left[ \int_0^1\bigl|Y^\hbar_{1,1,j}(t\by)\bigr|^4 dt\,\right]
\]
\[
=\vert \by\vert^4\int_0^1\bE_{\by,\bv}\Bigl[\,\bigl|Y^\hbar_{1,1,j}(t\by)\bigr|^4\,\Bigr] dt.
\]
We conclude that
\begin{multline}\label{det_had2}
g^\hbar(\bv,\by)\leq \sqrt{m} \vert \by\vert^2\left(\int_0^1\sum_{j=1}^m \bE_{\by,\bv}\Bigl[\,\bigl|Y^\hbar_{1,1,j}(t\by)\bigr|^4\,\Bigr] dt\right)^{\frac{1}{2}}\bE_{\by,\bv}\Bigl[\, \Vert\nabla^2 Y^\hbar(0)\Vert^{4(m-1)}\,\Bigr]^{\frac{1}{2}}.
\end{multline}
Step 2 will be completed once we prove the following result.

\begin{lemma}\label{lemma: gauss_est}  There exist $\hbar_2>0$, $r_2>0$ and  $C_2>0$  such that for any $ \hbar\leq h_1$  and any $\bv\in\bR^m$ we have
\begin{subequations}
\begin{equation}\label{det_had3a}
\bE_{\by,\bv}\Bigl[\, \Vert\nabla^2 Y^\hbar(0)\Vert^{4(m-1)}\,\Bigr]\leq C_2(1+\vert\bv\vert)^{4(m-1)},\;\;\forall \hbar<\hbar_2,\;\;\vert \by\vert<r_2,
\end{equation}
\begin{equation}\label{det_had3b}
\bE_{\by,\bv}\Bigl[\,\bigl|Y^\hbar_{1,1,j}(t\by)\bigr|^4\,\Bigr] \leq C_2(1+\vert\bv\vert)^4,\;\;\forall j,\;\; \hbar<\hbar_2,\;\;\vert \by\vert<r_2,\;\;t\in[0,1].
\end{equation}
\end{subequations}
\end{lemma}

\begin{proof}  The random matrix $\nabla^2 Y^\hbar$, conditioned by $C_\by(\bv)$, is \emph{Gaussian}. The same is true  of  $Y^\hbar_{1,1,j}(t\by)$ so it suffices to show  that there exist $\hbar_2>0$, $r_2>0$ and  $C_2>0$  such that for any $ \hbar\leq h_1$  and any $\bv\in\bR^m$ we have
\begin{subequations}
\begin{equation}\label{det_had4a}
\bE_{\by,\bv}\Bigl[\, \bigl|Y_{i,j}^\hbar(0)\bigr|^2\,\Bigr]\leq C_2(1+\vert\bv\vert)^{2},\;\;\forall i,j, \; \forall \hbar<\hbar_2,\;\;\vert \by\vert<r_2,
\end{equation}
\begin{equation}\label{det_had4b}
\bE_{\by,\bv}\Bigl[\,\bigl|Y^\hbar_{1,1,j}(t\by)\bigr|^2\,\Bigr] \leq C_2(1+\vert\bv\vert)^2,\;\;\forall j,\;\forall \hbar<\hbar_2,\;\vert \by\vert<r_2,\;\;t\in[0,1].
\end{equation}
\end{subequations}
As in \cite{Nifluct} we  introduce the index sets
\[
 \bsJ=\{\pm1,\pm 2,\dotsc, m\},\;\;J_\pm =\bigl\{ j\in\bsJ;\;\;\pm j>0\,\bigr\}.
 \]
 We consider the  $\bR^m\oplus \bR^m$ valued Random Gaussian vector  $G^\hbar(t)=G_-^\hbar\oplus G_+^\hbar$, $t\in[0,1]$, where
\[
G_-^\hbar:=\sum_{i=1}^m G^\hbar_{-i} \be_i=\nabla Y^\hbar(0),\;\;G_+^\hbar:=\sum_{j=1}^m G^\hbar_j\be_j=\nabla Y^\hbar(t\by).
\]
The covariance form   of this vector is the $2m\times 2m$ symmetric matrix 
\[
S_\hbar=S_\hbar(t)= \left[
\begin{array}{cc}
S^{-,-}_\hbar & S^{-,+}_\hbar\\
&\\
S^{+,-}_\hbar & S^{+,+}_\hbar
\end{array}
\right]=\left[
\begin{array}{cc}
A_\hbar & B_\hbar\\
B_\hbar & A_\hbar
\end{array}
\right],
\]
\[
A_\hbar=-\nabla^2 V^\hbar(0),\;\;B_\hbar=B_\hbar(t,\by)=-\nabla^2 V^\hbar(t\by).
\]
From (\ref{prob_densb}) we deduce that $S_\hbar$ is invertible if $\hbar<\hbar_1$ and $\vert \by\vert\leq r_1$. Its inverse has the block form
\[
\left[
\begin{array}{cc}
C_\hbar(t) &  -D_\hbar(t)\\
 -D_\hbar(t) & C_\hbar(t)
\end{array}
\right]=\left[
\begin{array}{cc}
C_\hbar(t) &  -A_\hbar^{-1}B_\hbar(t) C_\hbar(t)\\
 -A_\hbar^{-1}B_\hbar(t) C_\hbar(t) & C_\hbar(t)
\end{array}
\right],
\]
where,
\[
 C_\hbar(t)=C_\hbar(t,\by):= \bigl(\, A_\hbar- B_\hbar(t,\by) A_\hbar^{-1}B_\hbar(t,\by)\,\bigr)^{-1}.
\]
In \cite[Lemma 3.6]{Nifluct} we have shown  that there exists $\hbar_2\in (0,\hbar_1)$ such that the  $m\times m$ matrix
\[
K^\hbar:=(K^\hbar_{ij})_{1\leq i,j\leq m},\;\; K^\hbar_{ij}:= V^\hbar_{1,1,i,j}(0)
\]
is invertible for $\hbar\in[0, \hbar_2]$ and
\begin{equation}\label{degen1}
\lim_{t\to 0}  t^2  C_\hbar(t,\by) =(K^\hbar)^{-1},\;\;\mbox{\emph{uniformly} in $\hbar\in [0,\hbar_2]$ and $\vert \by\vert\leq r_1$
}.
\end{equation}
Next, observe that
\[
C_\hbar(t,\by)-D_\hbar(t,\by)= A^{-1}_\hbar\bigl(\, A_\hbar-B_\hbar(t,\by)\,\bigr)C_\hbar(t,\by)
\]
\[
=A_\hbar^{-1}\frac{1}{t^2}\bigl(\, A_\hbar-B_\hbar(t,\by)\,\bigr) t^2 C_\hbar(t,\by),
\]
and 
\[
\frac{1}{t^2}\bigl(\, A_\hbar-B_\hbar(t,\by)\,\bigr)=\frac{1}{t^2}\Bigl(\, \nabla^2 V^\hbar(t\by)-\nabla^2 V^\hbar(0)\,\Bigr)
\]
Since the function
\[
\bx\mapsto \nabla^2 V^\hbar(\bx)
\]
is even  and $V^\hbar\to V^0$ in the  $C^\infty$-topology  as $\hbar\to 0$ we deduce that the limit
\[
\lim_{t\to 0} \frac{1}{t^2}\Bigl(\, \nabla^2 V^\hbar(t\by)-\nabla^2 V^\hbar(0)\,\Bigr)
\]
exists, it is finite and it is \emph{uniform} in $\hbar\in[0,\hbar_2]$ and $\vert \by\vert\leq r_1$.  Using (\ref{degen1}) we conclude that there exists  a constant $c_1>0$ such that
\begin{equation}\label{degen2}
\Vert C_\hbar(t,\by)-D_\hbar(t,\by)\Vert \leq c_1,\;\;\forall t\in [0,1],\;\;\hbar\in[0,\hbar_2],\;\;\vert \by\vert\leq r_1.
\end{equation}
We can now prove  (\ref{det_had4a}) and (\ref{det_had4b}).

\smallskip

\noindent {\bf Proof of (\ref{det_had4a}).}  Fix $i_0,j_0\in \{1,\dotsc, m\}$.  The random variable $Y^\hbar_{i_0,j_0}(0)$,  conditioned by $C_\by(\bv)$, is  a normal random variable $\bar{Y}^\hbar_{i_0,j_0}$,  and its mean and variance are determined by the regression formula, \cite[Prop. 1.2]{AzWs}. To apply this formula we need to   compute the correlations between $Y^\hbar_{i_0,j_0}(0)$ and $G^\hbar$. These are given by the expectations
\[
\bXi_j^\hbar=\bXi^\hbar_j(\by)=\bE\bigl[\, Y^\hbar_{i_0,j_0}(0) G_j(1)\,\bigr] ,\;\;j\in\bsJ.
\]
We have
\[
\bXi_j^\hbar(\by)=\begin{cases}V^\hbar_{i_0,j_0, |j|}(0), & j\in\bsJ_-\\
V^\hbar_{i_0,j_0, j}(\by), &j\in\bsJ_+.
\end{cases}
\]
We regard the collection  $(\bXi_j^\hbar(\by))_{j\in\bsJ}$ as a linear map
\[
\bXi^\hbar(\by):\bR^m\oplus \bR^m\to\bR,\;\;  \bXi^\hbar\bigl(\,(z_j)_{j\in\bsJ}\,\bigr)=\sum_{j\in\bsJ} \bXi^\hbar_j(\by) z_j.
\]
In particular, we think of $\bXi^\hbar$ as a \emph{row} vector so its transpose $(\bXi^\hbar)^\top$ is a \emph{column} vector.

Observe that since  the function $V^\hbar$ is even, the third order derivative $V^\hbar_{ijk}$ are odd functions. Thus 
\[
V^\hbar_{i_0,j_0,j}(0)=0
\]
and there exists $r_2\in(0,r_1)$ and $c_2>0$ such that
\[
|V^\hbar_{i,j,k}(\by)|\leq c_2\vert \by\vert,\;\;\forall i,j,k,\;\forall \hbar\in [0,\hbar_2],\;\vert \by\vert\leq r_2.
\]
Hence
\begin{equation}\label{degen3}
\Vert \bXi^\hbar(\by)\Vert\leq c_2\vert \by\vert,\;\forall \hbar\in [0,\hbar_2],\;\vert \by\vert\leq r_2.
\end{equation}
Denote by $\hat{\bv}\in\bR^m\oplus \bR^m$ the vector $\bv\oplus\bv$.

According to the regression formula, the mean  of the conditioned random variable $\bar{Y}^\hbar_{i_0,j_0}$ is 
\[
\bE\bigl[\, \bar{Y}^\hbar_{i_0,j_0}\,\bigr]= -\bXi^\hbar(\by)\bigl(\, S_\hbar^{-1}\hat{\bv}\,\bigr)= \bXi^\hbar(\by)\bigl(\,  (C_\hbar-D_\hbar)\bv \oplus  (C_\hbar-D_\hbar)\bv\,\bigr).
\]
Using (\ref{degen2}) and (\ref{degen3}) we deduce that there exists $c_3>0$ such that
\begin{equation}\label{degen4}
\Bigl|\bE_{\by,\bv}\bigl[\, Y^\hbar_{i_0,j_0}\,\bigr]\Bigr|\leq c_3\vert \by\vert  \vert\bv\vert,\;\;\forall \hbar\in [0,\hbar_2],\;\;\vert \by\vert\leq r_2,\;\bv\in\bR^m.
\end{equation}
According to the regression formula, the variance  of the conditioned random variable $\bar{Y}^\hbar_{i_0,j_0}$ is 
\[
\var\bigl[\, \bar{Y}^\hbar_{i_0,j_0}\,\bigr]=\var\bigl[\,Y^\hbar_{i_0,j_0}\,\bigr]-\bXi^\hbar(\by) S_\hbar^{-1}(\bXi^\hbar(\by))^\top
\]
\[
= V^\hbar_{i_0,j_0,i_0,j_0}(0) -\bXi^\hbar(\by) S_\hbar^{-1}(\bXi^\hbar(\by))^\top
\]
Using (\ref{eq: vve_appr}), (\ref{degen2}) and (\ref{degen3}) we deduce that there exists $c_4>0$ such that
\begin{equation}\label{degen5}
\var_{\by,\bv}\bigl[\, Y^\hbar_{i_0,j_0}\,\bigr]\leq c_4  \vert\bv\vert,\;\;\forall \hbar\in [0,\hbar_2],\;\;\vert \by\vert\leq r_2,\;\bv\in\bR^m.
\end{equation}
The inequality (\ref{det_had4a}) now follows from (\ref{degen4}) and (\ref{degen5}).

\smallskip

\noindent {\bf Proof of (\ref{det_had4b}).} Fix $j_0\in\{1,2,\dotsc, m\}$ The random variable $Y^\hbar_{1,1,j_0}(t\by)$, conditioned by $C_{\by}(\bv)$ is a normal random variable $\bar{Y}^\hbar_{1,1,j_0}$.  To describe its mean and its variance we need to compute the correlations
\[
\bOm^\hbar_j(t,\by):=\bE\bigl[ \, Y^\hbar_{1,1,j_0}(t\by) G^\hbar_j\,\bigr]=\begin{cases}-V^\hbar_{1,1,j_0, |j|}(t\by), & j\in\bsJ_-\\
-V^\hbar_{1,1,j_0, j}\bigl( (1-t)\by\,\bigr), &j\in\bsJ_+.
\end{cases}
\]
Again,  we think of the collection  $(\bOm^\hbar_j(t,\by))_{j\in\bsJ}$ as defining a linear map
\[
\bOm^\hbar: \bR^m\oplus \bR^m\to\bR.
\]
The row vector $\bOm^\hbar$ splits as a direct sum of row vectors
\[
\bOm^\hbar=\bOm^\hbar_-\oplus \bOm^\hbar_+,\;\;\bOm^\hbar_\pm= \bigl(\, \bOm^\hbar_j\,\bigr)_{j\in \bsJ_\pm}.
\]
The mean of the random variable $\bar{Y}^\hbar_{1,1,j_0}$ is
\[
\bE\bigl[\, \bar{Y}^\hbar_{1,1,j_0}\,\bigr]=-\bOm^\hbar\bigl(\,  (C_\hbar-D_\hbar)\bv \oplus  (C_\hbar-D_\hbar)\bv\,\bigr).
\]
We conclude as befgore that
\begin{equation}\label{degen6}
\Bigl|\, \bE\bigl[\, {Y}^\hbar_{1,1,j_0}(t\by)\,\bigr]\,\Bigr|= O\bigl(\vert\bv\vert\cdot\vert \by\vert\bigl),\;\;\vert \by\vert\leq r_2,\;\;\bv\in\bR^m,\;\;t\in[0,1],
\end{equation}
 where  \emph{the constant implied  by the $O$-symbol is independent of $\hbar$ and $t$.  This  convention  will stay in place for the remainder of this proof.}
 
 Next, we have
 \[
 \var[\, \bar{Y}^\hbar_{1,1,j_0}\,\bigr]= \var[\, {Y}^\hbar_{1,1,j_0}(t\by)\,\bigr]-\bOm^\hbar S_\hbar^{-1} (\bOm^\hbar)^\top
 \]
 \[
 =V^\hbar_{1,1,j_0,1,1,j_0}(t\by)-\bigl[\,\bOm_-^\hbar\;\;\bOm^\hbar_+\,\bigr] \cdot\left[
\begin{array}{cc}
C_\hbar &  -D_\hbar\\
 -D_\hbar & C_\hbar
\end{array}
\right]\cdot\left[
\begin{array}{c}
(\bOm^\hbar_-)^\top\\
(\bOm^\hbar_+)^\top
\end{array}
\right]
\]
\[
=V^\hbar_{1,1,1,1,j_0,j_0}(t\by)-\bigl[\,\bOm_-^\hbar\;\;\bOm^\hbar_+\,\bigr] \cdot\left[
\begin{array}{c}
C_\hbar(\bOm^\hbar_-)^\top-D_\hbar(\bOm^\hbar_+)^\top\\
-D_\hbar(\bOm^\hbar_-)^\top+C_\hbar(\bOm^\hbar_+)^\top
\end{array}
\right]
\]
\[
= V^\hbar_{1,1,1,1,j_0,j_0}(t\by)-\Bigl(  \bOm_-^\hbar C_\hbar(\bOm^\hbar_-)^\top-\bOm_-^\hbar D_\hbar(\bOm^\hbar_+)^\top -\bOm^\hbar_+ D_\hbar(\bOm^\hbar_-)^\top+ \bOm^\hbar_+ C_\hbar(\bOm^\hbar_+)^\top\,\bigr)
\]
($D_\hbar=C_\hbar+O(1)$)
\[
=V^\hbar_{1,1,1,1,j_0,j_0}(t\by)-\Bigl(  \bOm_-^\hbar C_\hbar(\bOm^\hbar_-)^\top-\bOm_-^\hbar C_\hbar(\bOm^\hbar_+)^\top -\bOm^\hbar_+ C_\hbar(\bOm^\hbar_-)^\top+ \bOm^\hbar_+ C_\hbar(\bOm^\hbar_+)^\top\,\bigr) +O(1)
\]
\[
= V^\hbar_{1,1,1,1,j_0,j_0}(t\by)-\Bigl(\,  (\bOm_-^\hbar -\bOm^\hbar_+\,\bigr)C_\hbar(\bOm^\hbar_-)^\top -\bigl(\,\bOm_-^\hbar - \bOm^\hbar_+\,\bigr) C_\hbar(\bOm^\hbar_+)^\top\,\bigr) +O(1)
\]
\[
=  V^\hbar_{1,1,1,1,j_0,j_0}(t\by)- (\bOm_-^\hbar -\bOm^\hbar_+\,\bigr)C_\hbar (\bOm_-^\hbar -\bOm^\hbar_+\,\bigr)^\top+O(1).
\]
Now observe that $\Phi^\hbar=\bOm_-^\hbar -\bOm^\hbar_+$ is the vector in $\bR^m$   with components
\[
\Phi^\hbar_j(t,\by)=  V^\hbar_{1,1,j_0, j}\bigl( (1-t)\by\,\bigr)- V^\hbar_{1,1,j_0, j}\bigl( (t \by\,\bigr),
\]
that satisfy
\[
\bigl|\, \Phi^\hbar_j(t,\by)\,\bigr| =O(\vert \by\vert),\;\;\forall j.
\]
From (\ref{degen1}) we deduce
\[
\Vert C_\hbar(t,\by)\Vert= O\bigl(\vert \by\vert^{-2}\,\bigr)
\]
so that 
\[
\bigl|\, (\bOm_-^\hbar -\bOm^\hbar_+\,\bigr)C_\hbar (\bOm_-^\hbar -\bOm^\hbar_+\,\bigr)^\top\,\bigr|= O(1).
\]
This completes the proof of (\ref{det_had4b}) and  thus of Lemma \ref{lemma: gauss_est} and of statement (i) in Proposition \ref{prop: key}.
\end{proof}

\smallskip

\noindent {\bf Step 3.} The map 
\[
\bv\mapsto \bE\bigl[\, \bsZ^\hbar(\bv, B)\bigl(\,\bsZ^\hbar(\bv, B)-1\bigr) \,\bigr]
\]
is continuous.  This  follows by using the  argument in \emph{Point 2} in the proof of \cite[Prop. 1.1]{EL}.  Combined with (\ref{KR1}) will prove the statement (ii) in Proposition \ref{prop: key}.  

\smallskip

\noindent {\bf Step  4.}  Prove the  statement (iii) in Proposition \ref{prop: key}.   This  follows by using the  argument in \emph{Point 3} in the proof of \cite[Prop. 1.1]{EL}. 

\smallskip

\noindent {\bf Step 5.}    Using the  results  in {\bf Step 1} and {\bf Step 2}  and the dominated convergence theorem we obtain the statement (iv). \qed

\section*{Notation}

\begin{itemize}

\item We set 
\[
\bN:=\bigl\{n\in\bZ;\;\;n>0\,\bigr\},\;\;\bN_0:=\bigl\{n\in\bZ;\;\;n\geq 0\,\bigr\}.
\]
%\item If $\bsV$ is a real vector space, we denote by $\eL(\bsV)$ the space of  linear operators $\bsV\to\bsV$.

\item $\bone_A$ denotes the characteristic function of a subset $A$ of a set $S$,
\[
\bone_A: S\to\{0,1\},\;\;\bone_A(a)=\begin{cases}
1, & a\in A,\\
0, &a\in S\setminus A.
\end{cases}
\]
%\item For $k\in\bN_0\cup\{\infty\}$ and $n\in\bN$ we denote by  $C^k_b(\bR^n)$ the space of  $C^k$-functions with  bounded derivatives   of order $\leq k$ .

\item For a topological space $X$ we denote by $\eB(X)$ the $\si$-algebra of Borel subsets of $X$.

\item We will write $N\sim \eN(m,v)$ to indicate that  $N$ is a  normal random variable with mean $m$ and variance $v$.

\item For $\bx,\by\in\bR^m$ we set
\[
|\bx|_\infty:=\max_{1\leq j\leq m}|x_j|,\;\;(\bx,\by)=\sum_{j=1}^m x_jy_j,\;\;\vert \bx\vert:=\sqrt{(\bx,\bx)}.
\]
\item We denote by $\bA^m$ the \emph{affine lattice}
\[
\bA^m=\left(\frac{1}{2}+\bZ\right)^m.
\]
\item   For any  matrix $A$, we denote by $A^\top$ its transpose, and by $\Vert A\Vert$ its norm
\[
\Vert A\Vert=\sup_{|\bx|=1}|A\bx|.
\]

\item  We denote by $\one_m$ the  identity operator $\bR^m\to\bR^m$.

\item For any Borel subset $B\subset \bR^m$ we denote by $|B|$ its  Lebesgue measure.

\item We denote by $\bgamma$ the  canonical Gaussian measure  on $\bR$
\[
\bgamma(dx)=\frac{1}{\sqrt{2\pi}} e^{-\frac{x^2}{2}} dx,
\]
and by $\bGamma$   the  canonical Gaussian measure  on $\bR^m$
\[
\bGamma(d\bx)=(2\pi)^{-\frac{m}{2}} e^{-\frac{\vert x\vert^2}{2}} d\bx,
\]

\item  If $C$ is a symmetric, nonnegative  definite $m\times m$ matrix,    we write $N\sim\eN(0,C)$ to indicate that $N$  is an $\bR^m$-valued  Gaussian random vector with mean $0$ and covariance form $C$.

\item  If $f:\bR^m\to \bR$ is a twice differentiable function,  and $\bx\in\bR^m$, then we denote by  $\nabla^2 f(\bx)$  its  \emph{Hessian}, viewed as a symmetric operator $\bR^m\to\bR^m$.
\end{itemize}

\end{document}